\documentclass[12pt,twoside,leqno,final]{amsart}
\usepackage{amsmath}
\usepackage{amssymb}
 \usepackage{xcolor}
\theoremstyle{plain}
\newtheorem{thm}{Theorem}[section]
\newtheorem{lem}[thm]{Lemma}
\newtheorem{defn}[thm]{Definition}
\newtheorem{prop}[thm]{Proposition}
\newtheorem{cor}[thm]{Corollary}
\newtheorem{rem}[thm]{Remark}

\newcommand{\C}{{\mathbb C}}
\newcommand{\R}{{\mathbb R}}
\newcommand{\N}{{\mathbb N}}
\newcommand{\Z}{{\mathbb Z}}
\newcommand{\B}{{\mathbb B}}
\newcommand{\D}{{\mathbb D}}
\newcommand{\bB}{{\mathbb S}}

\newcommand{\Be}{\mathcal B}
\newcommand{\Ca}{\mathcal C}
\newcommand{\Qo}{\mathcal Q}

\newcommand{\p}{\partial}
\newcommand{\z}{\zeta}
\newcommand{\Chi}{\mathcal{X}}

\title[Sharp norm estimates for the Bergman operator]
{Sharp norm estimates for the Bergman operator from weighted mixed-norm spaces to weighted Hardy  spaces}

\author{Carme Cascante}
\address{Dept.\ Matem\`atica Aplicada i An\`alisi, Universitat  de Barcelona, Gran Via 585, 08071 Barcelona, Spain}
\email{cascante@ub.edu}

\author{Joan F\`abrega}
\address{Dept.\ Matem\`atica Aplicada i An\`alisi, Universitat  de Barcelona, Gran Via 585, 08071 Barcelona, Spain}
\email{joan$_{-}$fabrega@ub.edu}

\author{Joaqu\'\i n M. Ortega}
\address{Dept.\ Matem\`atica Aplicada i An\`alisi, Universitat  de Barcelona, Gran Via 585, 08071 Barcelona, Spain}
\email{ortega@ub.edu}


\date{\today}
\thanks{Partially supported by  DGICYT Grant MTM2011-27932-C02-01  and DURSI Grant 2014SGR 289.}

\begin{document}

\begin{abstract} 
In this paper we give sharp norm estimates for the Bergman operator acting from weighted 
mixed-norm spaces to weighted Hardy spaces in the ball,  
endowed with natural norms.
\end{abstract}

\maketitle

\section{Introduction}
The study of weighted norm inequalities for the Hardy-Littlewood maximal operator and for  singular  operators 
 in $\R^n$  and their relation with $A_p$-weights 
goes back to the works of Hunt, Muckenhoupt, Wheeden, Coiffman and Fefferman in the 70's. 
More recently,  many authors have studied the sharp dependence of the constants in  these  weighted norm inequalities. 
For the Hardy-Littlewood maximal operator, the weighted $L^p(\omega)$-norm 
is bounded, up to a  constant,  by $[\omega]_{A_p}^{1/(p-1)}$ 
and  the exponent is sharp in the sense that it can not be replaced by any smaller one (see \cite{Bu}).  For the more difficult 
case of singular integral operators
 in Euclidean spaces,  the so called $A_2$-conjecture for general Calderon-Zygmund operators was solved in \cite{Hy},
that is, the $L^2(\omega)$-norm of these operators are bounded, up to a constant, by $[\omega]_{A_2}$. 
Later, using a different approach,  a simple  and very elegant proof of the $A_2$-conjecture was given  in \cite{Le2}. 
The  sharp dependence on the weight has also been studied for other operators like the 
Lusin square function on $\R^n$ (see \cite{Le0} 
and the references therein). The proof of this result is based on 
the intrinsic square function of \cite{Wi} and strongly relies, 
among other key ingredients, on properties of convolution operators. 
Recently, an alternative proof has been  given in  \cite{Le}, which also proves the sharp dependence 
on the fixed aperture of the square function.  In the context of homogeneous spaces, some of these results have been extended 
 for the Hardy-Littlewood maximal operator and Calderon-Zygmund operators (see for instance \cite{Ka} and \cite{An-Va}  respectively).

A characterization of the measures on the unit ball $\B$ of $\C^n$  
for which the Bergman operator is bounded from the weighted $L^p$-space to the weighted Bergman space has been proved in \cite{Bekolle}. 
The condition is given in terms of the so called $B_p$-class. Later, sharp estimates on the norm of the Bergman operator on these weighted spaces were obtained in \cite{Po-Re}.
On the other hand, the boundedness of the Bergman operator acting 
from a mixed-norm space  to the Hardy space $H^p$ was proved in \cite{Co}.

The main objective of this paper is to obtain sharp estimates of the Bergman operator 
acting on weighted mixed-norm operators to weighted Hardy spaces, endowed with different norms.
The fact  that we are in a homogeneous space framework 
makes that  usual techniques, as convolution, can not be applied. In particular, 
it is necessary to use the specific properties of the kernels involved in the 
different problems we consider, as well as  the dyadic decomposition for homogeneous spaces obtained in \cite{Ka},
which gives the existence of adjacent and sparse families of cubes.

Before we state our main results, we recall some definitions.

Let $H=H(\B)$ be the space of holomorphic functions on the unit ball of $\C^n$, $\B$. 
We denote by $H^*$ the space of functions $f\in H$ which have boundary values 
$f(\z)=\lim_{r\nearrow 1} f(r\z)$ a.e. on the unit sphere $\bB$.

For $0<p<\infty$ and  $\omega$ a weight on $\bB$, that is a function 
$\omega\in L^1$ satisfying  $\omega>0$ a.e., let $L^p(\omega)=L^p(\bB,\omega d\sigma)$, 
where $d\sigma$ denotes the normalized Lebesgue measure on the unit sphere $\bB$. 
If $\omega=1$, we use $L^p$ to denote $L^p(\bB,d\sigma)$.  

Denote by 
$$
H^p(\omega)=\{f\in H^*: \|f\|_{H^p(\omega)}=\|f\|_{L^p(\omega)}<\infty\}.
$$

If $1< p<\infty$ and $\omega$ in the Muckenhoupt class $A_p=A_p(\bB)$, 
which will be defined in the Section \ref{sec:preliminaries},  
there exist other characterizations of the space $H^p(\omega)$
(see for instance \cite[Section 5]{Lu}, \cite{Ca-Or4} and the references therein). 
In this paper, we consider the equivalent norm given in terms of the Littlewood-Paley function.

Let $L^{p,2}(\omega)$ be the mixed-norm space of measurable functions $\varphi$ on $\B$ satisfying 
$$
\|\varphi\|_ {L^{p,2}(\omega)}^p
=\int_{\bB}\left(\int_0^1|\varphi(r\z)|^2\frac{ 2nr^{2n-1} dr}{1-r^2}\right)^{p/2}\omega(\z)d\sigma(\z)<\infty,
$$

We denote by $F^{p,2}_0(\omega)$ the weighted Triebel-Lizorkin space of holomorphic functions $f$ on $\B$ satisfying 
\begin{equation}\label{eqn:TL}
\|f\|_{F^{p,2}_0(\omega)}= \left\|(1-r^2)\left(I+\frac{R}{n}\right)f(r\z)\right\|_{ L^{p,2}(\omega)}<\infty.
\end{equation}
Here $I$ denote the identity operator and  $R=\sum_{j=1}^n z_j\frac{\p}{\p z_j}$.
If $\omega=1$, we simply write $H^p$ and $F^{p,2}_0$.
In this case, it is well know that for $0<p<\infty$ these spaces are isomorphic  and 
this isomorphism still holds if we replace the operator  $I+\frac{R}{n}$ by $\alpha I+\beta R$ with $\alpha,\beta>0$ 
(see for instance  \cite{Pav}, \cite{Ah-Br} and the references therein).
If $1<p<\infty$ and $\omega\in A_p$, these results are also true for the spaces $H^p(\omega)$ and 
$F^{p,2}_0(\omega)$ (see \cite{Ca-Or4}).
In particular, there exist constants $c,C>0$ depending on $p$, $n$ and $\omega$, such that 
$$c\|f\|_{H^p(\omega)}\leq \|f\|_{F^{p,2}_0(\omega)}\leq C \|f\|_{H^p(\omega)}.$$

Denote by $\Ca$ the Cauchy integral operator on $L^p$ and by $\Be$ the Bergman integral operator on 
$L^p(d\nu)=L^p(\B,d\nu)$, $p\ge 1$, given by
$$
\Ca(\psi)(z)=\int_{\bB}\psi(\z)\Ca(z,\z)d\sigma(\z),\quad\text{and}\quad \Be(\varphi)(z)=\int_{\B}\varphi(w)\Be(z,w)d\nu(w),
$$
where
$$ 
\Ca(z,\zeta)=\frac{1}{(1-z\overline \z)^{n}}\quad\text{and}\quad 
\Be(z,w)=\frac{1}{(1-z\overline w)^{n+1}}=\left(I+\frac{R}{n}\right)\Ca(z,w).
$$
Here, $d\nu$ denotes the normalized Lebesgue measure on $\B$.

In \cite{Hu-Mu-Whe} and \cite{Lee-Rim}, the authors proved that  if $1<p<\infty$, 
then the Cauchy operator is bounded on $L^p(\omega)$  if and only if $\omega\in A_p$. 
Thus, it is then natural to consider this problem for   the norm $\|\Ca:L^p(\omega)\to F^{p,2}_0(\omega)\|$ 
with weights $\omega\in A_p$, $p>1$, where we recall that $F^{p,2}_0(\omega)$ is normed by \eqref{eqn:TL}.
Using adequate pairings, it  is easy to check  (see for instance \cite{Co} for the unweighted case and 
Section \ref{sec:preliminaries} in general) that the adjoint operator of 
$(1-|z|^2)\left(I+\frac{R}{n}\right)\Ca:L^p(\omega)\to L^{p,2}(\omega)$ is the Bergman operator
 $\Be:L^{p',2}(\omega')\to H^{p'}(\omega')$, where $p'$ is the conjugate exponent of $p$ and $\omega'=\omega^{1-p'}$.

The main result of this paper is the following:

\begin{thm}\label{thm:Bergman}
Let $1<p<\infty$ and let $\omega $ be a weight on $\bB$. Then, $\Be$ is a bounded operator from 
$L^{p,2}(\omega)$ to $H^p(\omega)$  if and only if $\omega \in A_p$. 
In this case, we have 
\begin{equation}\label{eqn:Bergman}
\|\Be :L^{p,2}(\omega)\to H^p(\omega)\|\leq C(p,n) [\omega]_{A_p}^{\max\{1,1/(2(p-1))\}}
\end{equation}
and the estimate is sharp.
\end{thm}
Throughout the paper, a sharp estimate will mean that the exponent of $[\omega]_{A_p}$ can not be replaced by a smaller one.

For the unweighted case we have:

\begin{thm} \label{cor:Bergman}
If $1<p<\infty$, then $\|\Be :L^{p,2}\to H^p\|\leq C(n)\max\{p,\sqrt{p'}\}$. This estimate is also sharp.
\end{thm}

One natural question that arise from Theorem \ref{thm:Bergman}, is the study of the norm of  the operator 
$\Be:L^{p,2}(\omega)\to F^{p,2}_0(\omega)$.
For this case we have:

\begin{thm} \label{thm:Bergman2}
Let $1<p<\infty$ and $\omega $ a weight on $\bB$. Then, $\Be$ is a bounded operator from 
$L^{p,2}(\omega)$ to $F^{p,2}_0(\omega)$ if and only if $\omega\in A_p$.

In this case, we have that 
$$
\|\Be:L^{p,2}(\omega)\to F^{p,2}_0(\omega)\| \le C(p,n) [\omega]_{A_p}^{\max\{1,1/(p-1)\}}.
$$ 
\end{thm}

If the weight $\omega$ is in a more regular class, then it can be obtained a sharper estimate 
of the norm of the operator. Namely, we have the following result:

\begin{thm} \label{thm:Bergman3}
Let $1<p<\infty$.
\begin{enumerate}

	\item If $p>2$ and  $\omega\in A_{1}$, then 
$$
\|\Be:L^{p,2}(\omega)\to F^{p,2}_0(\omega)\|\le C(p,n) [\omega]_{A_1}^{1/2}.
$$

\item If $1<p<2$ and $\omega' \in A_{1}$, then 
$$
\|\Be:L^{p,2}(\omega)\to F^{p,2}_0(\omega)\|\le C(p,n) [\omega']_{A_1}^{1/2}
$$
 In both cases, the estimate is  sharp.
 \end{enumerate}
 \end{thm}
 
For the unweighted case, we obtain:
\begin{thm} \label{cor:Bergman3}
If $1<p<\infty$, then $\|\Be :L^{p,2}\to F^{p,2}_0\|\leq C(n)\max\{\sqrt{p},\sqrt{p'}\}$. This estimate is also sharp.
\end{thm}

The paper is organized as follows. In Section \ref{sec:preliminaries}, 
we recall the results on Muckenhoupt weights, duality and extrapolation theorems that will be needed in the proof of our main results. 
In Section \ref{sec:necessary}, it is proved the necessary condition $\omega\in A_p$ 
 in Theorems \ref{thm:Bergman} and \ref{thm:Bergman2}. In Section \ref{sec:estimate}, 
 we prove the estimate \eqref{eqn:Bergman} in Theorem \ref{thm:Bergman}. 
Its sharpness will be proved  in Section \ref{sec:sharpness}. In this section we also prove Theorem \ref{cor:Bergman}.
 Finally, Sections \ref{sec:thmbergman2} and \ref{sec:thmbergman3} are devoted to the proof of Theorems 
\ref{thm:Bergman2}, \ref{thm:Bergman3} and \ref{cor:Bergman3}.

A final remark on notations. If $X$ and $Y$ are a couple of normed spaces and $T$ is a bounded linear operator from  $X$ to $Y$, 
we denote its norm by $\|T:X\to Y\|$. If $X=Y$ we also denote this norm by $\|T\|_X$.

Throughout the paper 
$C(p_1,\cdots,p_k)$ will denote a positive constant depending only of the parameters $p_1,\cdots,p_k$, 
which may vary from place to place.
If we do not need to track  the dependence of the constant, we write  $f\lesssim g$  to denote 
the existence of a constant $C$ such that $f\le C g$ and $f\approx g$ to denote $f\lesssim g\lesssim f$.

\section{Preliminaries}\label{sec:preliminaries}
\subsection{The Muckenhoupt class $A_p$}\quad\par
In this section we recall the definition and some properties of the weights in $A_p$.
\begin{defn} 
We say that a nonnegative function $\omega\in L^1$ is in the Muckenhoupt class $A_p$, $1<p<\infty$, if 
$$
[\omega]_{A_p}=\sup_{B}\frac{\omega(B)\left(\omega'(B)\right)^{p/p'}}{|B|^p}<\infty,
$$
where the supremum is taken over all nonisotropic balls $B$ 
$$
B=B(\zeta,r)=\{\eta\in\bB:\,|1-\z\overline \eta|<r\},
$$
$\omega'=\omega^{-(p'-1)}=\omega^{-p'/p}$
and
$$
\omega(B)=\int_{B}\omega d\sigma.
$$
Here, if $E\subset \bB$ is measurable, we write $|E|=\sigma(E)$.
\end{defn}

\begin{defn}
 A nonnegative function $\omega\in L^1$ is in $A_1$ if there exists $C>0$ such that for a.e. $\zeta\in\bB$,
 $M[\omega](\zeta)\leq C \omega(\zeta)$, where
 if $\psi\in L^1$, we denote by $M(\psi)$ the nonisotropic Hardy-Littlewood maximal function defined by
$$M(\psi)(\zeta)=\sup_{\zeta\in B}\frac1{|B|}\int_B |\psi| d\sigma.$$

Here,
$[\omega]_{A_1}=\operatorname{ess}\sup_{\zeta\in \bB}\frac{M(\omega)(\zeta)}{\omega(\zeta)}.$
\end{defn}

The following property of the weights is well known (see \cite{Du2}).
\begin{lem}
\begin{enumerate}
	\item If $1\leq p\leq q<+\infty$, then:
 $A_p\subset A_q$  and  $[\omega]_{A_q}\le [\omega]_{A_p}$.
\item If $\omega\in A_p$, then $\omega'\in A_{p'}$ and $[\omega']_{A_{p'}}=[\omega]_{A_p}^{p'-1}=[\omega]_{A_p}^{1/(p-1)}$.
\end{enumerate}
\end{lem}

\subsection{The spaces $L^{p,2}(\omega)$}\quad\par

It is well known that if $\mu$ is a positive measure on a set $X\subset \C^n$, then for $1<p<\infty$, the dual of $L^p(\mu)$ can be identified with 
$L^{p'}(\mu)$, in the sense that, for each  $\Gamma\in (L^{p}(\mu))'$ there exists a unique $\psi\in L^{p'}(\mu)$ such that
$\Gamma(\varphi)=\int_X \varphi\psi d\mu$ and $\|\Gamma\|= \|\psi\|_{L^{p'}(\mu)}$.

Using this fact, we have that if $1<p<\infty$ and $\omega\in A_p$, then  for any  linear form $\Gamma\in (L^{p}(\omega))'$ 
there exists a unique $\psi\in L^{p'}(\omega')$ such that 
$$
\Gamma(\psi)=\langle \varphi,\psi\rangle_\bB=\int_{\bB}\varphi \overline \psi d\sigma,
$$
and moreover,  $\|\Gamma\|= \|\psi\|_{L^{p'}(\omega')}$.
That is, the dual of $L^p(\omega)$ with respect to the pairing $\langle \cdot,\cdot\rangle_\bB$ is $L^{p'}(\omega')$.

Since the Cauchy projection maps $L^p(\omega)$ onto $H^p(\omega)$,  with the same pairing we can identify 
the dual of $H^p(\omega)$ with $H^{p'}(\omega')$ for $p>1$ (see \cite{Lu}).

An analogous duality result for mixed-norm spaces is proved in \cite{BenedekPanzone}, which restricted to our case is:

\begin{prop}\label{prop:dualLpq}
 Let $1<p<\infty$ and $\omega\in A_p$. The dual of the mixed-norm space $L^{p,2}(\omega)$ 
with respect to the pairing 
$$
\langle \varphi,\psi\rangle_\B=\int_{\B}\varphi(z)\overline{\psi(z)}\frac{d\nu(z)}{1-|z|^2}
=\int_{\bB}\int_0^1\varphi(r\z)\overline{\psi(r\z)}\frac{2n r^{2n-1}\,d r}{1-r^2} d\sigma(\z)
$$
is $L^{p',2}(\omega')$. 
That is, for any $\Gamma \in \left(L^{p,2}(\omega) \right)'$ 
there exists $\psi\in L^{p',2}(\omega')$ such that 
$\Gamma(\varphi)=\langle \varphi,\psi\rangle_\B$ and $\|\Gamma\|= \|\psi\|_{L^{p',2}(\omega')}$.
\end{prop}

\subsection{The estimate of the Hardy-Littlewood maximal operator}
We recall that in \cite{Bu} it was obtained a norm-estimate for the  Hardy-Littlewood maximal operator $M$ 
on weighted Lebesgue spaces on $\R^n$. This result was extended to  metric spaces with a doubling measure.

\begin{thm}{\cite[Proposition 7.13]{HyKa}} \label{thm:bucley} 
If  $1<p<\infty$ and $\omega\in A_p$, then 
$$
\|M\|_{L^p(\omega)} \lesssim [w]^{1/(p-1)}_{A_p}.
$$
\end{thm}

\subsection{An extrapolation theorem}
It is shown in \cite{Du} a  version of the extrapolation theorem of Rubio de Francia, which will be used in the proof of our results. Namely, we have:
\begin{thm}[\cite{Du}]\label{thm:duoandikoetxea}
Assume that for some family of pairs of nonnegative functions, $(\varphi,\psi)$, for some $p_0\in [1,\infty)$, 
and for all $\omega\in A_{p_0}$, we have
$$
\left( \int_{\bB} \psi^{p_0}\omega d\sigma\right)^{1/p_0}
\leq C N([w]_{A_{p_0}}) \left( \int_{\bB} \varphi^{p_0}\omega d\sigma\right)^{1/p_0},
$$
where $N$ is an increasing function and the constant $C$ does not depend on $\omega$. 
Then for all $1<p<\infty$ and all $\omega\in A_p$ we have
$$
\left( \int_{\bB} \psi^{p}\omega d\sigma \right)^{1/p}\leq C K(w) \left(  \int_{\bB} \varphi^p\omega  d\sigma \right)^{1/p},
$$
where
$$
K(\omega)=\begin{cases}
N\left([\omega]_{A_p} (2\|M\|_{L^p(\omega)})^{p_0-p}\right),&\text{ if $p<p_0$}\\
N\left([\omega]_{A_p}^{(p_0-1)/(p-1)}  (2\|M\|_{L^{p'}(\omega')})^{(p-p_0)/(p-1)}\right),&\text{ if $p>p_0$}.
\end{cases}
$$

In particular, $K(w) \leq C_1 N((C_2[\omega]_{A_p}^{\max\{1,(p_0-1)/(p-1)\}})$.
\end{thm}

 \begin{rem}
 This theorem is proved in \cite{Du} in $\R^n$, but it can be easily extended to the setting 
of homogeneous spaces using Theorem \ref{thm:bucley}.
 \end{rem}

\section{Proof of the necessary condition in Theorems \ref{thm:Bergman} and  \ref{thm:Bergman2}}\label{sec:necessary}

Since 
\begin{equation}
\|\Be:L^{p,2}(\omega)\to F^{p,2}_0(\omega)\|= \|\Qo:L^{p,2}(\omega)\to L^{p,2}(\omega)\|
\end{equation}
where   $\Qo(\varphi)(z)=(1-|z|^2)\left(I+\frac{R}{n}\right)\Be(\varphi)(z)$,
the necessary condition $\omega\in A_p$ in Theorems   \ref{thm:Bergman} and  \ref{thm:Bergman2}, follows from the following proposition.

\begin{prop} \label{prop:neccond}
Let $1<p<\infty$ and let $0\le \omega\in L^1$. 
 If either $\Be$ is bounded from $L^{p,2}(\omega)$ to $L^p(\omega)$, or  $\Qo$ 
is bounded from $L^{p,2}(\omega)$ to itself, then $\omega\in A_p$.
\end{prop}

\begin{proof}
	The proof of this proposition follow using standard arguments (see for instance \cite{Co-Fe} or 
	\cite{Bekolle}) and, for a sake of completeness, we will give a sketch of it.

For $0\ne a\in \B$, let $a^*=a/|a|$,   $B_a=\{\z\in\bB:\,|1-\z\overline a^*|<1-|a|\}$ and let $S_a$ be the nonisotropic square
$$
S_a=\{w=s\eta \in\overline{\B}: 1-s\le 1-|a|, |1-\eta \overline a^*|\le 1-|a|, \eta\in \bB\}.
$$

Note that if $w=s\eta\in S_a$, then
$$
1-|a|\le |1-w\overline a|\le 1-|a|+1-s+|1-\eta\overline {a^*}|\le 3(1-|a|).
$$
Since $d(z,w)=|1-z\overline w|^{1/2}$ satisfies the triangle inequality (see \cite[Proposition 5.1.2]{Ru}), 
 for  $\kappa>0$ large enough, there exists $0<r_\kappa<1$  such that:
 for each $a\in\B$, $|a|>r_\kappa$, there exists $b\in\B$  satisfying $|a|=|b|$,  $|1-b\overline a|=\kappa (1-|a|)$ and 
$$
|1-z\overline{w}|\approx |1-z\overline{a}|\approx \kappa (1-|a|)
\approx \kappa |1-w\overline{a}|,\quad  \text{ for any $w\in S_a$ and $z\in S_b$},
$$
where the constants in the equivalences do not depend on $z,w,a,b$ and $\kappa$.
Thus, 
\begin{align*}
|\Be(z,w)-\Be(z,a)|&\le c(n)\frac{|1-w\overline a|^{1/2}}{|1-z\overline a|^{n+3/2}}
\le \frac{c'(n)}{\sqrt{\kappa}}\frac{1}{|1-z\overline a|^{n+1}}.
\end{align*}

So, choosing $\kappa\ge (2 c'(n))^2$,  for any $0\le \varphi \in L^1(d\nu)$ we have 
\begin{equation} \label{eqn:Be1}
\Chi_b(z) |\Be(\Chi_a\varphi)(z)|\ge \frac{1}{2}\frac{\Chi_b(z)}{|1-z\overline a|^{n+1}}\int_{S_a}\varphi d\nu
\approx\frac{\Chi_b(z)}{(1-|a|^2)^{n+1}}\int_{S_a}\varphi d\nu.
\end{equation}
where $\Chi_a$ and $\Chi_b$ denote the characteristic function of $S_a$ and $S_b$, respectively, 
and the constants in the last equivalence depend only on $n$ and $\kappa$.

Analogously we have 
\begin{equation}\label{eqn:Be2} 
\Chi_b(z)\frac{1-|z|^2}{(1-|a|^2)^{n+2}}\int_{S_a}\varphi d\nu\lesssim  \Chi_b(z)|\Qo(\Chi_a\varphi)(z)|,
\end{equation}
where, as above,  the constants in the inequality depend only on $n$ and $\kappa$.

Let $\psi\ge 0$ be a continuous function on $\overline\B$ and for $s>0$, 
let $\varphi(s\eta)=(1-s)\Chi_a(\eta)\psi(\eta)\in L^{p,2}(\omega)$. 
Then, by integration in polar coordinates and using that $1-|a|=1-|b|$, we have
\begin{align*}
\int_\B \Chi_a(w) \varphi(w)d\nu(w) &\approx (1-|a|^2)^2\int_{\bB}\Chi_a(\eta)\psi(\eta) d\sigma(\eta),\quad\\
 \|\Chi_a \varphi\|_{L^{p,2}(\omega)} &\approx (1-|a|)\int_\bB\Chi_a(\eta)\psi(\eta) \omega(\eta) d\sigma(\eta),\\
\int_{\bB}\Chi_b(\eta)  \omega(\eta) d\sigma(\eta) &\approx \omega(B_b) \quad\text{and}\\
\|(1-|z|^2)\Chi_b(z)\|_{L^{p,2}(\omega)} &\approx (1-|a|)\omega(B_b)^{1/p}.
\end{align*}

Therefore,  \eqref{eqn:Be1} and \eqref{eqn:Be2} give  
\begin{align*}
\frac{\omega(B_b)^{1/p}}{(1-|a|^2)^n}\int_{B_a}\psi d\sigma
&\lesssim  \|\Be:L^{p,2}(\omega)\to H^p(\omega)\| \left(\int_{B_a}\psi^p\omega d\sigma\right)^{1/p}\text{and}\\
\frac{\omega(B_b)^{1/p}}{(1-|a|^2)^{n}}\int_{B_a}\psi d\sigma
&\lesssim  \|\Be:L^{p,2}(\omega)\to F^{p,2}_0(\omega)\| \left(\int_{B_a}\psi^p\omega d\sigma\right)^{1/p}.
\end{align*}
These inequalities applied to the function $\psi=1$ gives $\omega(B_b)\lesssim \omega(B_a)$. 
Interchanging $a$ and $b$ we also obtain $\omega(B_b)\approx \omega(B_a)$. 
Hence,  in both cases for any $|a|>r_\kappa$ 
and $\psi$ a continuous function on $\bB$, we have
$$
\left(\frac{1}{\sigma(B_a)}\int_{B_a}\psi d\sigma \right)^p\lesssim \frac{1}{\omega(B_a)}\int_{B_a}\psi^p\omega d\sigma.
$$

Since $\bB$ is the finite union of sets $B_{a_j}$, $|a_j|>r_\kappa$, and the space of continuous 
functions on $\bB$ is dense in $L^p(\omega)$,  the above inequality holds for any $B_a$ 
and any $\psi\in L^p(\omega)$. 
This  is equivalent to $\omega\in A_p$ (see for instance \cite[p.195]{Stein}).
\end{proof}

\begin{rem}\label{rem:contraexem}
It is well known that for the the boundedness of  $\Be$ on $L^p(\B)$  is equivalent to the  boundedness on $L^p(\B)$ of the integral operator  $|\Be|$ associated to the kernel $|\Be(z,w)|$. In our situation, even for $n=1$, we have that the integral operator  $|\Be|$ is not bounded from $L^{p,2}$ to $L^p$. 
For instance, consider the function $\varphi(w)=\left(\log\frac{2}{1-|w|^2}\right)^{-1}\in L^{p,2}$. We have,
\begin{align*}
\int_{\D}\frac{\left(\log\frac{2}{1-|w|^2}\right)^{-1}}{|1-z\overline w|^2}d\nu(w)
&\gtrsim \int_0^1 r \left(\log\frac{2}{1-r^2}\right)^{-1} \int_0^{1}\frac{dt}{(1-|z|^2+1-r^2+t)^2}dr\\
&\gtrsim \int_0^1 \frac{r\left(\log\frac{2}{1-r^2}\right)^{-1} }{1-|z|^2+1-r^2}dr\gtrsim \log\log\frac{2}{1-|z|^2}.
\end{align*}
Consequently, $|\Be|(\varphi)$ has not boundary values.
\end{rem}

\section{Proof of the estimate \eqref{eqn:Bergman} in Theorem \ref{thm:Bergman}}\label{sec:estimate}
In Proposition \ref{prop:neccond}, we have proved that if the Bergman operator 
is bounded from $L^{p,2}(\omega)$ to $H^p(\omega)$, then $\omega\in A_p$. 
In order to finish the proof of Theorem \ref{thm:Bergman}, we first observe that condition \eqref{eqn:Bergman} 
can be rewritten by duality as an estimate of the weighted Triebel-Lizorkin norm of the Cauchy operator. Namely,
we have the following result:

\begin{prop}\label{lem:adjBergman}
If $1<p<\infty$, we have that the following conditions are equivalent:
\begin{enumerate}
\item\label{enum:bergman1}
 For any  $\omega\in A_p$ and any $\varphi\in L^{p,2}(\omega)$, 
 \begin{equation}\label{eqn:Bergman3}
\|\Be(\varphi)\|_{H^p(\omega)}\leq C(p,n) [\omega]_{A_p}^{\max\{1,1/(2(p-1))\}}\|\varphi\|_{L^{p,2}(\omega)}.
\end{equation}
\item\label{enum:bergman2} For $\omega\in A_p$ and any $\psi \in L^{p}(\omega)$,
\end{enumerate}
\begin{equation}\label{eqn:adjointlerner}\begin{split}
\|{\mathcal C}(\psi)\|_{F^{p,2}_0(\omega)}
&=\left\|(1-|z|^2)\left(I+\frac{R}{n}\right){\mathcal C}(\psi)(z)\right\|_{L^{p,2}(\omega)}\\
&\leq C(p,n) [\omega ]^{\max\{1/2,1/(p-1)\}}\|\psi\|_{L^{p}(\omega)}.
\end{split}
\end{equation}
Moreover, $\|\Be: L^{p,2}(\omega)\to L^{p}(\omega)\|=\|\Ca:L^{p'}(\omega')\to F_0^{p',2}(\omega')\|$.
\end{prop}
\begin{proof}
Since $$\left(I+\frac1{n}R\right)\frac1{(1-z\overline{\zeta})^n}=\frac1{(1-z\overline\zeta)^{n+1}},$$ 
for  any smooth functions $\varphi$ and $\psi$ on $\B$ and $\bB$, respectively,  Fubini's Theorem gives that
$$
\langle \Be(\varphi),\psi\rangle_{\bB}
=\left\langle \varphi(z),(1-|z|^2)\left(I+\frac1{n}R\right){\Ca}(\psi)(z)\right\rangle_\B.
$$  

Hence, \eqref{eqn:Bergman3} is equivalent to 
\begin{equation}\label{eqn:adjointlerner1}\begin{split}
\|{\mathcal C}(\psi)\|_{F^{p',2}_0(\omega')}
&=\left\|(1-|z|^2)\left(I+\frac{R}{n}\right){\mathcal C}(\psi)(z)\right\|_{L^{p',2}(\omega')}\\
&\leq C(p,n) [\omega']^{\max\{1/2,1/(p'-1)\}}\|\psi\|_{L^{p'}(\omega')},
\end{split}
\end{equation}
which is also equivalent to \eqref{enum:bergman2}.
\end{proof}

Observe that the key estimate \eqref{eqn:adjointlerner1} is a non isotropic version of Theorem 1.1 in \cite{Le0}, 
which is based in the intrinsic square function introduced in \cite{Wi}. 
The original proof heavily relies in the convolution in $\R^n$. In our situation, 
there is no such convolution and we instead follow closely 
some of the main ideas in Theorem 1.1 in \cite{Le}.

Although we state our main results for the operator $\left(I+\frac{R}{n}\right)\Ca$, 
all the norm-operator estimates also hold for any operator 
$\left(\alpha I+\beta R\right)\Ca$, $\alpha,\beta \in\R$ 
(see Remark \ref{rem:equivnormsC} below).

\subsection{Preliminary results}\quad\par

In the proof of our main results we will use the dyadic decomposition 
of a quasi-metric space of \cite{Chr} (see also \cite{HyKa}  and \cite{Ka}).
We recall that $\rho$ is a quasi-metric on a space $X$ if it satisfies the axioms 
of a metric except for the triangle inequality, 
wich is assumed in a weaker form: there exists $A_0\geq 1$ such that for any $x,y,z\in X$, 
$\rho(x,y)\leq A_0(\rho(x.z)+\rho(z,y))$. 
The quasi-metric space $(X,\rho)$ is also assumed to satisfy the following geometric doubling property:
 there exists $N\in\N$ 
such that for every $x\in X$ and for every $r>0$, the ball $B(x,r)=\{y\in X;\, \rho(x,y)<r\}$ 
can be covered by at most $N$ balls $B(x_i,r/2)$.
 We will state the decomposition for $\bB$ and the quasi-metric $\rho(\zeta,\eta)=|1-\zeta\overline{\eta}|$. 
Observe that $A_0=2$. 

\begin{prop}\label{prop:dyadicdec} Given a fixed parameter $0<\delta<1$, 
small enough and a fixed point $x_0\in\bB$, 
there exists  a finite collection of families of sets, ${\mathcal D}^j$, $j=1,\dots,M$, called the adjacent dyadic systems, 
such that each ${\mathcal D}^j$ is a family of Borel sets $Q_\alpha^k$, $k\in\Z$, $\alpha\in{I}_k$, 
called the dyadic cubes, 
which are associated with points $\zeta_\alpha^k$, which we will call the center points of the cubes $Q_\alpha^k$, 
having the following properties:
\begin{enumerate}
\item\label{item:dec.1} $\bB=\cup_{\alpha\in I_j} Q_\alpha^k$ (disjoint union), for each $k\in \Z$.
\item\label{item:dec.2} if $k<l$, then either $Q_\beta^l\cap Q_\alpha^k=\emptyset$ or $Q_\beta^l\subset Q_\alpha^k$.
\item\label{item:dec.3}  There exist $c_1,C_1>0$ such that 
$B(\zeta_\alpha^k,c_1\delta^k)\subset Q_\alpha^k\subset B(\zeta_\alpha^k,C_1\delta^{k})= B(Q_\alpha^k).$
\item\label{item:dec.4} if $k\leq l$ and $Q_\beta^l\subset Q_\alpha^k$, then $B(Q_\beta^l)\subset B(Q_\alpha^k)$.
\item\label{item:dec.5}  For any $k\in\Z$, there exists $\alpha$ such that $x_0=\zeta_\alpha^k$, the center point of $Q_\alpha^k$.
\item\label{item:dec.6} There exists $C>0$ (only depending on $A_0$ and $\delta$) such that for any nonisotropic ball 
$B(\zeta,r)\subset \bB$, with $\delta^{k+3}<r\leq \delta^{k+2}$, there exists $j$ and $Q_\alpha^k\in{\mathcal D}^j$ such that
$B(\zeta,r)\subset Q_\alpha^k$ and ${\bf diam}\,\, Q_\alpha^k\leq Cr$.
\end{enumerate}

 The family ${\mathcal D}=\bigcup_{j=1}^M {\mathcal D}^j$ is called a dyadic decomposition of $\bB$ 
and we say that the set $Q_\alpha^k$ is a dyadic cube of generation $k$ centered at $\zeta_\alpha^k$ 
with radius $l(Q_\alpha^k)=\delta^k$.
 \end{prop}
 \begin{rem} It is immediate to check that from properties  \eqref{item:dec.3}, \eqref{item:dec.1} and 
\eqref{item:dec.2},  that  there exists $\varepsilon>0$  (only depending on the dimension $n$ and on $\delta$)  and
for any $Q_1^k\in{\mathcal D}^j$ there exists at least one $Q_2^{k+1}\in {\mathcal D}^{j}$  
so that $Q_2^{k+1}\subset Q_1^k$ and 

  \begin{equation}\label{item:dec.8}|Q_2^{k+1}|\geq \varepsilon |Q_1^k|.\end{equation}
\end{rem}

Before we go back to the proof of Proposition \ref{lem:adjBergman}, 
we need to introduce some more notations and results. 
The non-increasing rearrangement of a measurable function $\psi$ on $\bB$ is defined by
$$\psi^*(t)=\inf \{\alpha>0\,;\, |\{\zeta\in\bB\,;\, |\psi(\zeta)|>\alpha\}|\leq t\}
=\sup_{C\subset\bB\,;|C|=t}\inf _{\zeta\in C} |\psi(\zeta)|,\qquad 0<t<\infty.$$
It is immediate to check that
$$|\{\zeta\in \bB\,;\, |\psi(\zeta)|>\lambda \}|=|\{ t>0\,;\, \psi^*(t)>\lambda \}|.$$

Let $\psi$ be a measurable function on $\bB$. 
If  $Q$ is a dyadic cube, the local mean oscillation of $\psi$ on $Q$ is given by
$$\omega_\lambda(\psi;Q)= \inf_{c\in \R} ((\psi-c)\Chi_Q)^*(\lambda|Q|),\qquad 0<\lambda<1.$$

We will denote by $m_Q(\psi)$, the median value of $\psi$ over $Q$, a (possibly non unique) real number such that
$$\max\left\{|\{\zeta\in Q\,;\, \psi(\zeta)>m_Q(\psi)\}|,|\{\zeta\in Q\,;\, \psi(\zeta)<m_Q(\psi)\}| \right\}\leq |Q|/2.$$
It is immediate to check that
\begin{equation}\label{eqn:median}
|m_Q(\psi)|\leq(\psi\Chi_Q)^*(|Q|/2).
\end{equation}
Next, given a dyadic cube $Q_0\in {\mathcal D}^j$, let us denote ${\mathcal D}^j(Q_0)$ the dyadic cubes of ${\mathcal D}^j$ 
contained in $Q_0$. The dyadic local sharp maximal function $m_{\lambda;Q_0}^{\#,d}\psi$ is defined by
\begin{equation}\label{eqn:sharp}
m_{\lambda;Q_0}^{\#}\psi(\zeta)=\sup_{\zeta\in Q'\in{\mathcal D}^j(Q_0)}\omega_\lambda(\psi;Q').
\end{equation}
It is also well known (see for instance \cite{JaTo}) that a.e $\zeta\in\bB$,
$$m_{\lambda;Q_0}^{\#}\psi(\zeta)\lesssim M[\psi](\zeta),$$

 If $Q_0\in{\mathcal D}^j$, a family of sets ${\mathcal S}(Q_0)$ is sparse in $Q_0$ with respect to the dyadic decomposition 
${\mathcal D}$, if
$$
{\mathcal S}(Q_0)=\cup_{m\geq 0}C_m,
$$
where
\begin{enumerate}
	\item Each $C_m$ is a family of sets in ${\mathcal D}^j$ which are subsets of $Q_0$.
	\item $C_0=\{ Q_0\}$.
	\item The elements of each family $C_m$ are pairwise disjoint.
	\item For any $m>0$, every $Q\in C_m$ is a subset of an element of $C_{m-1}$.
	\item For any $Q_1\in C_m$ we have that
	\begin{equation}\label{eqn:sparse}
	\left|\cup_{Q\in C_{m+1}} Q\cap Q_1\right| \leq \frac{|Q_1|}{2}.
	\end{equation}
\end{enumerate}
We denote \begin{equation}\label{eq:sparse}E_{Q_1}= Q_1\setminus \cup_{Q\in C_{n+1}} Q\cap Q_1.\end{equation}
We then have that $|E_{Q_1}|\geq|Q_1|/2$.

The proof of Proposition \ref{lem:adjBergman} is based in a homogeneous version of the key estimate in \cite{Le4}, that it is proved in \cite{An-Va}. Namely,

\begin{thm}\label{thm:homoglerner}

Let $\psi$ a measurable function on $\bB$ and $Q_0\in {\mathcal D}^j$ a fixed cube and $\varepsilon$ as in \eqref{item:dec.8}. Then there exists a (possibly empty) sparse family of cubes ${\mathcal S}(Q_0)$ 
such that for a.e. $\zeta\in Q_0$, 
\begin{equation}\label{eqn:homoglerner}
|\psi(\zeta)-m_{Q_0}(\psi)|\leq m^{\#}_{\varepsilon/4,Q_0}(\psi)(\zeta)
+ \sum_{Q\in {\mathcal S}(Q_0)} \omega_{\varepsilon/4}(\psi,Q)\Chi_Q(\zeta).
\end{equation}
\end{thm}

\subsection{Main estimate}\quad\par
We begin recalling some technical lemmas. 
The first one is a version of a Whitney decomposition of an open set in $\bB$, which can be found in \cite{Ca-Or}.
 \begin{lem}\label{lemma:nonisotropicwhitney}
 Let $R>1$ and let $\Omega$ be an open set in $\bB$. Consider a dyadic adjacent system ${\mathcal D}^j$ in $\bB$, 
$j\in\{1,\dots,M\}$. If $j$ is fixed, Let $  \Lambda^j$ be the family of cubes $Q_\alpha^k\in{\mathcal D}^j$, 
which are maximal with respect to the property $ RB(Q_\alpha^k)\subset \Omega$.
 We then have:
 \begin{enumerate}
 \item\label{item:whit1} $\Omega=\bigcup_{Q_\alpha^k\in \Lambda^j} Q_\alpha^k$ and for the cubes in 
$\Lambda^j$, either $Q_\alpha^k\cap Q_{\alpha_1}^{k_1}=\emptyset$ or $Q_\alpha^k= Q_{\alpha_1}^{k_1} $.
 \item\label{item:whit2} 
There exists $K>0$ only depending on the constants $C_1$ and $\delta$  
of the definition of the dyadic adjacent system (see Proposition \ref{prop:dyadicdec}), 
such that for  every $Q_\alpha^k\in \Lambda^j$, we have that  
$KRB(Q_\alpha^k)\cap \Omega^c\neq \emptyset$.
 \item\label{item:whit3}  
There exists $C(C_1,\delta)>0$, only depending on the constants $C_1$ and $\delta$ of the definition 
of the dyadic adjacent system, such that 
$$
\sum_{Q_\alpha^k\in \Lambda^j} \Chi_{RB(Q_\alpha^k)}\leq C(C_1,\delta) \Chi_\Omega.
$$
 \end{enumerate}
 \end{lem}

 In \cite{tch}, it is proved the following
\begin{lem}\label{eqn:3.1}
There exist $K_1,K_2>0$ such that for any $\zeta,\zeta',\xi\in\bB$, $\rho<1$, satisfying 
$|1-\zeta\overline{\xi}|\geq K_1 |1-\zeta\overline{\zeta}'|$, then
\begin{equation*}
 \left|\frac1{(1-\rho\zeta\overline{\xi})^{n+1}}-\frac1{(1-\rho\zeta'\overline{\xi})^{n+1}}\right|
\leq K_2\left( \frac{|1-\zeta\overline{\zeta'}|}{|1-\zeta\overline{\xi}|}\right)^\frac12\frac1{|1-\zeta\overline{\xi}|^{n+1}}.
 \end{equation*}
\end{lem}
 
 The  following lemma  is based in the well known technique of splitting functions of A.P. Calderon and A. Zygmund.
\begin{lem}\label{lem:wilson}
There exists $C>0$ such that for any $\lambda>0$, $\psi \in L^1$,
$$
\left|\left\{ \eta\in \bB;\, \left( \int_0^1 (1-r^2) 
\left| \int_{\bB} \frac{\psi(\zeta)}{(1-r\eta\overline{\zeta})^{n+1}}d\sigma(\zeta)\right|^2 dr\right)^{1/2} 
>\lambda\right\}\right| \lesssim \frac{\|\psi\|_{L^1}}{\lambda}.
$$
\end{lem}
\begin{proof}

 We denote by $G(\psi)$ the function on $\bB$ defined by
 $$G(\psi)(\eta)= \left(\int_0^1\left|\left(I+\frac{R}{n}\right){\Ca}(\psi)(r\eta)\right|^2(1-r^2)dr\right)^\frac12.$$
If $\lambda>0$ and $\psi\in L^1$, we denote $\Omega_\lambda= \{ \eta\in \bB;\,  M(\psi)(\eta)>\lambda\}$.

Since the nonisotropic Hardy-Littlewood maximal operator is of weak type $(1,1)$,
we have that
$$|\Omega_\lambda|\lesssim \frac1{\lambda}\int_{\bB} |\psi(\zeta)|d\sigma(\zeta).$$

We must then estimate
$$|\{ \eta\notin \Omega_\lambda;\, G(\psi) (\eta)>\lambda\}|.$$
By  Lemma \ref{lemma:nonisotropicwhitney}, there exists $(Q_k)_k$  a 
Whitney decomposition of the set $\Omega_\lambda$. 
We split $\psi$ into two pieces, $\psi=g+b$, where
$$g(\zeta)=\left\{\begin{array}{ll}
\psi(\zeta),\qquad \zeta\notin \Omega_\lambda\\
\frac1{|Q_k|}\int_{Q_k} \psi d\sigma\qquad \zeta\in Q_k.
\end{array}\right.$$
Property \eqref{item:whit2} of the Whitney decomposition give that $\|g\|_\infty\lesssim \lambda$.
Put $b_k= b\Chi_{Q_k}=(\psi-\psi_{Q_k})\Chi_{Q_k}$, where 
$\displaystyle{\psi_{Q_k}=\frac1{|Q_k|} \int_{Q_k}\psi d\sigma}$.
Then $b_k$ is supported in $Q_k$, $\int_{Q_k}b_k=0$ and $\|b_k\|_{L^1}\lesssim \int_{Q_k}|\psi|d\sigma$.
We also have that $b=\sum_k b_k$.

We decompose
\begin{align*}
\left|\{ \eta\notin \Omega_\lambda;\, G(\psi)(\eta) >\lambda\}\right|&\leq
\left|\{ \eta\notin \Omega_\lambda;\, G(g)(\eta) >\lambda/2\}\right|
+\left|\{ \eta \notin \Omega_\lambda;\, G(b)(\eta)>\lambda/2\}\right|\\
&=I+II.
\end{align*}
We will estimate each term separately.

For the first one we use  Chebyshev's inequality 
and the facts that both $\Ca$ and the Littlewood-Paley $g$-function are bounded on $L^2(\bB)$.
\begin{align*}
\left|\{ \eta\notin \Omega_\lambda;\, G(g)(\eta)>\lambda/2\}\right| 
&\lesssim \frac1{\lambda^2} \int_{\bB} G(g)(\eta)^2d\sigma(\eta)
\lesssim\frac1{\lambda^2}\int_{\bB} |g(\eta)|^2d\sigma(\eta)\\
&\lesssim\frac1{\lambda} \left( \int_{\bB\setminus \Omega_\lambda} |\psi(\zeta)|d\sigma(\zeta)
+ \int_{ \Omega_\lambda} |g(\zeta)|d\sigma(\zeta)\right)\\
&\lesssim \frac1{\lambda}\left( \int_{\bB\setminus \Omega_\lambda} |\psi(\zeta)|d\sigma(\zeta)
+ \sum_k \frac1{|Q_k|}  \int_{Q_k}|\psi(\zeta)|d\sigma(\zeta)\right)\\
&\lesssim \frac1{\lambda} \int_{\bB}|\psi(\zeta)|d\sigma(\zeta),
\end{align*}

We now estimate $II$. Let $\eta\in \bB\notin \Omega_\lambda$. Denote by $\xi_k$ the "center" of $Q_k$, $k\geq 1$. 
Since for each $k\geq 1$, $\int_{\bB} b_k=0$, we have
\begin{align*}
\int_{\bB} \frac1{(1-r\eta\overline{\zeta})^{n+1}} b_k(\zeta)d\sigma(\zeta)=
\int_{\bB} \left(\frac1{(1-r\eta\overline{\zeta})^{n+1}} 
-\frac1{(1-r\eta\overline{\xi_k})^{n+1}}\right) b_k(\zeta)d\sigma(\zeta).
\end{align*}
Next, observe that if  we choose $R$ in Lemma \ref{lemma:nonisotropicwhitney} such that for any $\zeta\in Q_k$ and
 $\eta \in\bB\notin \Omega_\lambda$, we have that 
$|1-r\eta\overline{\zeta}| \geq K_1|1-\zeta\overline{\xi}_k|$, where $K_1$ is as in Lemma \ref{eqn:3.1}. 
Thus, this lemma  gives that the above integral is bounded by
\begin{align*}
\int_{\bB} \frac{|1-\zeta\overline{\xi_k}|^{1/2}}{|1-r\eta\overline{\xi_k}|^{n+1+1/2}}|b_k(\zeta)|d\sigma(\zeta)
\lesssim \int_{\bB} \frac{l(Q_k)^{1/2}}{|1-r\eta\overline{\xi_k}|^{n+1+1/2}}|b_k(\zeta)|d\sigma(\zeta).
\end{align*}
But
$$
\left(\int_0^1(1-r^2) \frac{dr}{|1-r\eta\overline{\xi_k}|^{2n+3}} \right)^{1/2}
\lesssim \frac1{|1-\eta\overline{\xi_k}|^{n+1/2}},
$$
and, consequently,
\begin{align*}
\int_{\bB\setminus \Omega_\lambda} G(b_k)(\eta) d\sigma(\eta)
&\lesssim l(Q_k)^{1/2}\int_{Q_k} |b_k(\zeta)|d\sigma(\zeta) \int_{\bB\setminus \Omega_\lambda} 
\frac{d\sigma(\eta)}{|1-\eta\overline{\xi_k}|^{n+1/2}}\\
&\lesssim \int_{Q_k} |b_k(\zeta)|d\sigma(\zeta).
\end{align*}
Altogether, 
\begin{align*}
\int_{\bB\setminus \Omega_\lambda} G(b)(\eta) d\sigma(\eta)&\lesssim \sum_k  \int_{Q_k} |b_k(\zeta)|d\sigma(\zeta)\\
&\lesssim  \sum_k\int_{Q_k} |\psi(\zeta)|d\sigma(\zeta)\lesssim \int_{\bB}|\psi(\zeta)|d\sigma(\zeta).
\end{align*}
From this estimate we deduce immediately that
$$
\left|\{ \eta \notin \Omega_\lambda;\, G(b)(\eta)>\lambda/2\}\right|\lesssim \frac1{\lambda}\int_{\bB}|\psi(\zeta)d\sigma(\zeta)|.
$$
And that finishes the proof.
\end{proof}

 We  now can  prove the main lemma which is a version for the sphere with 
the nonisotropic distance $\rho$ of Lemma 3.1 in \cite{Le0}. 
Here, in our situation, we skip the fact that we do not have convolution, 
using  the estimate in Lemma \ref{eqn:3.1}:
 \begin{lem}\label{lem:mainestimate}
Let ${\mathcal D}^j$, $j=1,\dots, M$ an adjacent dyadic system in $\bB$ as in 
Proposition \ref{prop:dyadicdec} and let $0<\lambda<1$ be fixed. 
Then,  for $\psi\in L^1$ and for any cube $Q\in {\mathcal D}^j$, we have the estimate
\begin{equation*}
\omega_\lambda(G(\psi)^2;Q)
\lesssim \sum_{k\geq 0} \frac1{2^{k/2}}\left( \frac1{|2^k B(Q)|}\int_{2^kB(Q)} |\psi(\eta)|d\sigma(\eta)\right)^2.
\end{equation*}
\end{lem}
\begin{proof}
Let $K_1, K_2$ be as in Lemma \ref{eqn:3.1}. If $\zeta\in Q$, we decompose $G(\psi)^2(\zeta)$ in two terms given by
\begin{align*}G(\psi)^2(\zeta)&=\int_{1-4K_1l(Q)}^1 \left|\left(I+\frac{R}{n}\right){\Ca}(\psi)(r\zeta)\right|^2(1-r^2)dr\\
&+ \int_0^{1-4K_1l(Q)}\left|\left(I+\frac{R}{n}\right){\Ca}(\psi)(r\zeta)\right|^2(1-r^2)dr\\
&= I_1(\psi)(\zeta)+I_2(\psi)(\zeta).\end{align*}

We will first show that
\begin{equation}\label{eqn:mainestimate11}
\left(I_1(\psi)\Chi_Q\right)^*(\lambda|Q|) 
\lesssim \sum_{k\geq0}\frac1{2^{k}}\left( \frac1{|2^kB(Q)|}\int_{2^kB(Q)}|\psi(\eta)|d\sigma(\eta)\right)^2.
\end{equation}
Since $(x+y)^2\leq 2(x^2+y^2)$, we have that for any $\zeta\in Q$,
$$I_1(\psi)(\zeta)\leq 2\left( I_1(\psi\Chi_{4B(Q)})(\zeta)+I_1(\psi\Chi_{\bB\setminus 4B(Q)})(\zeta)\right),$$
and consequently,
\begin{equation}\label{eqn:mainestimate2}
\left(I_1(\psi)\Chi_Q\right)^*(\lambda|Q|) 
\lesssim \left(I_1(\psi\Chi_{4B(Q)})\right)^*(\lambda|Q|/2) 
+\left(I_1(\psi\Chi_{\bB\setminus 4B(Q)})\right)^*(\lambda|Q|/2).
\end{equation}
By Lemma \ref{lem:wilson} we have that
\begin{align*}
\left(I_1(\psi\Chi_{4B(Q)})\right)^*(\lambda|Q|/2)
&\leq \left((G(\psi\Chi_{4B(Q)}))^2\right)^*(\lambda|Q|/2)\\
&\lesssim \left(\frac1{|4B(Q)|} \int_{4B(Q)} |\psi(\eta)| d\sigma(\eta)\right)^2.
\end{align*}

For any $\eta \in \bB\setminus 4B(Q)$, $|1-\zeta\overline{\eta}|>l(Q)$. Hence,
\begin{align*}
 \left| \left(I+\frac{R}{n}\right)\Ca (\psi\Chi_{\bB\setminus 4B(Q)})(r\zeta)\right|
&\lesssim \int_{|1-\zeta\overline{\eta}|>l(Q)}\frac{|\psi(\eta)|}{((1-r)+|1-\zeta\overline{\eta}|)^{n+1}}d\sigma(\eta)\\
&\lesssim \sum_{k\geq 0}\frac1{(2^kl(Q))^{n+1}}\int_{2^{k}B(Q)}|\psi(\eta)|d\sigma(\eta)\\
&\approx\frac1{l(Q)}\sum_{k\geq 0}\frac1{2^k}\frac1{|2^{k}B(Q)|}\int_{2^{k}B(Q)}|\psi(\eta)|d\sigma(\eta).
\end{align*}

Thus, for any $\zeta\in Q$
\begin{align*}
& I_1(\psi\Chi_{\bB\setminus 4B(Q)})(\zeta)\\
&\lesssim\int_{1-4K_1l(Q)}^1 \frac1{l(Q)^2}
\left( \sum_{k\geq0} \frac1{2^k}\frac1{|2^kB(Q)|} \int_{2^kB(Q)}|\psi(\eta)|d\sigma (\eta)\right)^2(1-r^2)dr\\
&\approx\left( \sum_{k\geq0} \frac1{2^k}\frac1{|2^kB(Q)|} \int_{2^kB(Q)}|\psi(\eta)|d\sigma(\eta) \right)^2.
\end{align*}
By Chebyshev's inequality,
\begin{align*}
\left(I_1(\psi\Chi_{\bB\setminus 4B(Q)})\Chi_Q\right)^*(\lambda|Q|/2)
&\lesssim \frac{\| I_1(\psi\Chi_{\bB\setminus 4B(Q)})\Chi_Q\|_{L^1}}{(\lambda|Q|)/2}\\
&\lesssim\frac{|Q|}{\lambda|Q|}\left(\sum_{k\geq0} \frac1{2^k}  \frac1{|2^kB(Q)|} 
\int_{2^kB(Q)} |\psi(\eta)|d\sigma(\eta)\right)^2,
\end{align*}
and consequently, applying Schwartz's inequality,
\begin{equation}\label{eqn:mainestimate1}\begin{split}
&\left(I_1(\psi\Chi_{\bB\setminus 4B(Q)})\Chi_Q\right)^*(\lambda|Q|/2)\\
&\lesssim \sum_{k\geq0} \frac1{2^k} \left( \frac1{|2^kB(Q)|} \int_{2^kB(Q)} |\psi(\eta)|d\sigma(\eta)\right)^2,
\end{split}\end{equation}
which finishes the proof of the estimate \eqref{eqn:mainestimate11}.

In order to estimate $\omega_\lambda(G(\psi)^2;Q)$, consider any $\zeta_2\in \bB$.
Observe that
\begin{align*}
\omega_\lambda(G(\psi)^2;Q)
&\leq \left( (G(\psi)^2-I_2(\psi)(\zeta_2))\Chi_Q\right)^*(\lambda|Q|)\\
&\lesssim \left(I_1(\psi\Chi_Q)\right)^*((\lambda|Q|/2)
+ \left( (I_2(\psi)-I_2(\psi)(\zeta_2))\Chi_Q\right)^*((\lambda|Q|/2))\\
&\lesssim \left( I_1(\psi)\Chi_Q\right)^*((\lambda|Q|)/2)+\|I_2(\psi)-I_2(\psi)(\zeta_2)\|_{L^\infty(Q)}.
\end{align*}

 So we are left to estimate $\|I_2(\psi)-I_2(\psi)(\zeta_2)\|_{L^\infty(Q)}$.
Let $\zeta_1,\zeta_2\in Q$. Then,
\begin{align*}
&|I_2(\psi)(\zeta_1)-I_2(\psi)(\zeta_2)|\\
&=\left|\int_0^{1-4K_1l(Q)} \left(\left| \left(I+\frac{R}{n}\right)\Ca(\psi)(r\zeta_1)\right|^2
-\left| \left(I+\frac{R}{n}\right)\Ca(\psi) (r\zeta_2)\right|^2 \right) (1-r^2)dr \right|.\end{align*}

But 
\begin{align*}
&\left|\left|\int_{\bB}\frac{\psi(\eta)}{(1-r\zeta_1\overline{\eta})^{n+1}}d\sigma(\eta)\right|^2
-\left|\int_{\bB}\frac{\psi(\eta)}{(1-r\zeta_2\overline{\eta})^{n+1}}d\sigma(\eta)\right|^2\right|\\
&\leq
\left|  \int_{\bB} \left| \frac{1}{(1-r\zeta_1\overline{\eta})^{n+1}}-\frac{1}{(1-r\zeta_2\overline{\eta})^{n+1}}  \right| 
|\psi(\eta)|d\sigma(\eta)\right|\\
&\times\left( \left|\int_{\bB} \frac{|\psi(\eta)|}{|1-r\zeta_1\overline{\eta}|^{n+1}}d\sigma(\eta)\right|
+\left|\int_{\bB} \frac{|\psi(\eta)|}{|1-r\zeta_2\overline{\eta}|^{n+1}}d\sigma(\eta)\right|\right).
\end{align*}
Now, if $\zeta_1,\zeta_2\in Q$, $|1-\zeta_1\overline{\zeta_2}|\leq 4l(Q)$, and consequently, then
for any $0<r<1-4K_1l(Q)$ and any $\eta\in \bB$, $|1-r\overline{\eta}\zeta_1|\geq K_1|1-\zeta_1\overline{\zeta_2}|$, 
and consequently, applying Lemma \ref{eqn:3.1},
\begin{align*}
\left| \frac{1}{(1-r\zeta_1\overline{\eta})^{n+1}}-\frac{1}{(1-r\zeta_2\overline{\eta})^{n+1}} \right| 
&\lesssim \left(\frac{|1-\zeta_1\overline{\zeta_2}|}{|1-r\zeta_1\overline{\eta}|}\right)^\frac12 
\frac1{|1-r\zeta_1\overline{\eta}|^{n+1}}\\
&\lesssim \frac{l(Q)^\frac12}{(1-r)^\frac12|1-r\zeta_1\overline{\eta}|^{n+1}}.
\end{align*}
As a consequence, since  
$|1-r\zeta_1\overline{\eta}|\approx |1-r\zeta_2\overline{\eta}|$,

\begin{align*}
&\left|  I_2(\psi)(\zeta_1)-I_2(\psi)(\zeta_2)\right|\\
&\lesssim l(Q)^\frac12\int_0^{1-4K_1l(Q)} 
\left( \int_{\bB} \frac{|\psi(\eta)|}{|1-r\zeta_1\overline{\eta}|^{n+1}}d\sigma(\eta)\right)^2 (1-r^2)^\frac12 dr\\
&\lesssim {\sum_{k\geq2}}'\int_{1-2^{k+1}K_1l(Q)}^{1-2^kK_1l(Q)}l(Q)^\frac12 \left( \int_{2^kB(Q)} 
\frac{|\psi(\eta)|}{|1-r\zeta_1\overline{\eta}|^{n+1}}d\sigma(\eta)\right)^2 (1-r^2)^\frac12 dr\\
&+{\sum_{k\geq2}}'\int_{1-2^{k+1}K_1l(Q)}^{1-2^kK_1l(Q)}l(Q)^\frac12 \left( \int_{\bB\setminus 2^kB(Q)} 
\frac{|\psi(\eta)|}{|1-r\zeta_1\overline{\eta}|^{n+1}}d\sigma(\eta)\right)^2 (1-r^2)^\frac12 dr\\
&=J_{21}+J_{22}.
\end{align*}
Here, by $\sum_{k\geq2}'$ we mean that the summands are considered only for 
those $k\ge 2$ such that $2^{k+1}K_1l(Q)<1$
We begin with the estimates of $J_{21}$.
\begin{align*}
J_{21} &\lesssim {\sum_{k\geq2}}'l(Q)^\frac12 \left(2^kl(Q)\right)^\frac32 \left(  \frac1{(2^kl(Q))^{n+1}} 
\int_{2^kB(Q)}|\psi(\eta)|d\sigma(\eta)\right)^2\\
&\lesssim{\sum_{k\geq2}}'\frac1{2^{k\frac12}} \left(  \frac1{|2^kB(Q)|} \int_{2^kB(Q)}|\psi(\eta)|d\sigma(\eta)\right)^2.
\end{align*}
Next,
\begin{align*}
J_{22} &\lesssim {\sum_{k\geq2}}' l(Q)^\frac12 \left(2^kl(Q)\right)^\frac32 
\left( \sum_{i>k} \frac1{(2^i l(Q))^{n+1}} \int_{2^iB(Q)} |\psi(\eta)|d\sigma(\eta) \right)^2\\
&\lesssim {\sum_{k\geq2}}' 2^{\frac{k}2} \sum_{i\geq k} \frac1{2^i} 
\left( \frac1{|2^iB(Q)|} \int_{2^iB(Q)} |\psi(\eta)|d\sigma(\eta)\right)^2\\
&={\sum_{i\geq 2}}' \sum_{k\leq i} 2^{\frac{k}2} \frac1{2^i} \left( \frac1{|2^iB(Q)|} 
\int_{2^iB(Q)} |\psi(\eta)|d\sigma(\eta)\right)^2\\
&\lesssim {\sum_{i\geq 2}}' \frac1{2^{\frac{i}{2}}} \left( \frac1{|2^iB(Q)|} \int_{2^iB(Q)}|\psi(\eta)|d\sigma(\eta)\right)^2.
\end{align*} 

Finally, \eqref{eqn:mainestimate1} and the above estimates finish the proof of the lemma.
\end{proof}

\subsection{Proof of the estimate in Theorem \ref{thm:Bergman}}\quad\par
By Proposition \ref{lem:adjBergman}, it is enough to show that for  any $\omega\in A_p$,  
$\| {\mathcal C}(\psi)\|_{F_0^{p,2}(\omega)} \leq C(p,n)[\omega]_{A_p}^{\max \{1/2,1/(p-1)\}}\|\psi\|_{L^{p}(\omega)}$. 
As we have recalled at the beginning of this section, the proof of Theorem \ref{thm:Bergman} follows closely 
the ideas in \cite{Le} and in \cite{Le3}. We now sketch how to finish the proof. 
First, by Lemma \ref{lem:mainestimate}, we have that 
a.e. $\zeta\in Q$, $m_{\lambda, Q}^{\#}G(\psi)^2(\zeta)\lesssim M(\psi)(\zeta)^2$.  
Next, we have that for any $Q\in {\mathcal D}^i$, 
there exists a sparse family ${\mathcal S}(Q)=\left(Q_j^k \right)$, $Q_j^k\in{\mathcal D}^i$ so that if we denote by 
$$
{\mathcal T}_l^{\mathcal S}(\psi)(\zeta)
= \left( \sum_{Q_j^k\in{\mathcal S}(Q)} (\psi_{2^lB(Q_j^k)})^2\Chi_{Q_j^k}(\zeta) \right)^{1/2},
$$
then by Theorem \ref{thm:homoglerner} and our previous observation, we have that for a.e $\zeta\in Q$,
 
\begin{equation}\label{eqn:lernerthm1}
|G(\psi)(\zeta)^2-m_Q(G(\psi)^2)|\lesssim \left( M(\psi)(\zeta)^2
+ \sum_{l\geq 0} \frac1{2^{l/2}} \left( {\mathcal T}_{l}^{\mathcal S}(\psi)\right)^2\right).
\end{equation}

Hence
\begin{equation}\label{eqn:lernerthm2}
|G(\psi)(\zeta)^2-m_Q(G(\psi)^2)|^{1/2}\lesssim  M(\psi)(\zeta)+ {\mathcal T}^{\mathcal S}(\psi)(  \zeta),
\end{equation}
where 
$$
{\mathcal T}^{\mathcal S}(\psi)(  \zeta)= \sum_{l^\geq 0} \frac1{2^{l/4}} {\mathcal T}_{l}^{\mathcal S}(\psi)(\zeta).
$$

The following lemma gives an estimate for the first term ${\mathcal T}_0^{\mathcal S}$.
It was originally proved 
 in \cite{CrU-Ma-Pe} for $\R^n$. For a sake of completeness, we give an alternative proof, much simpler, 
obtained in \cite{Le3}, adapted for our setting of homogeneous spaces. 
 \begin{lem}\label{thm.lernerdec} For each $\psi\in L^1(Q)$ and $\omega\in A_3$,
$$\|{\mathcal T}_{0}^{\mathcal S}(\psi)\|_{L^3(\omega)} \lesssim  [\omega]_{A_3}^{ 1/2}\|\psi\|_{L^3(\omega)},$$
with constant independent of the family ${\mathcal S}$.
\end{lem}

\begin{proof}
Since $\|{\mathcal T}_{0}^{\mathcal S}(\psi)\|_{L^3(\omega)} 
=\|{\mathcal T}_{0}^{\mathcal S}(\psi)^2\|_{L^{3/2}(\omega)}^{1/2}$. 
Using duality, it is enough that we show that for any $\varphi\geq 0$, $\|\varphi\|_{L^{3}(\omega)}=1$
\begin{align*}
&\int_Q \left( {\mathcal T}_0^{\mathcal S}(\psi)(\eta)\right)^2 \varphi(\eta)\omega(\eta) d\sigma(\eta)\\
& =\sum_{Q_j^k\in {\mathcal S}(Q)} \left( \frac1{|B(Q_j^k)|} \int_{B(Q_j^k)}|\psi(\eta)|d\sigma(\eta)\right)^2
\int_{Q_j^k} \varphi(\eta)\omega(\eta) d\sigma(\eta)\\
&\lesssim [\omega]_{A_3} \|\psi\|_{L^3(\omega)}^2.
\end{align*}

We next denote by $T_3(E)=\frac{\omega(E)(\omega^{-1/2}(E))^2}{|E|^3}$. 
The sparsity of the family $(Q_j^k)_{j,k}$ gives that there exists sets $(E_{Q_j^k})_{j,k}$  satisfying that are pairwise disjoint,  and $|E_{Q_j^k}|\gtrsim |Q_j^k|$ (see  \eqref{eq:sparse}). 
Hence, using that $\omega\in A_3$, we have that there exists $A>0$ such that
\begin{align*}&
\left( \frac1{| B(Q_j^k)|}\int_{B(Q_j^k)} |\psi(\eta)|d\sigma(\eta)\right)^2 \int_{Q_j^k} \varphi(\eta)\omega(\eta) d\sigma(\eta) \\
&\lesssim T_3(A B(Q_j^k)) \left( \frac1{\omega^{-1/2}(AB(Q_j^k))} \int_{B(Q_j^k)}|\psi|d\sigma\right)^2\left( \frac1{\omega(AB(Q_j^k))} 
\int_{B(Q_j^k)} \varphi\omega d\sigma \right)|E_{Q_j^k}|\\&
\lesssim [\omega]_{A_3} \int_{E_{Q_j^k}} (M_{\omega^{-1/2}}(\psi\omega^{1/2}))^2 M_\omega (\varphi)d\sigma.
\end{align*}
 
Here $M_{\omega}$ denotes the weighted Hardy Littlewood maximal function defined by
$$M_\omega(\varphi)(\zeta)=\sup_{\zeta\in B}\frac1{\omega(B)}\int_{B} |\varphi|\omega d\sigma.$$
Since $\omega\in A_3$, $w^{-1/2}$ is in $A_{3/2}$ and we have that both $\omega$ and $\omega^{-1/2}$ satisfy a doubling condition. 
Hence both weighted maximal functions are of strong type (see, for instance \cite{Ka}), and using H\"older's inequality, 
the sum of the above estimates can be bounded as follows:
\begin{align*}&
\sum_{Q_j^k\in{\mathcal S}(Q)}\left( \frac1{| B(Q_j^k)|}\int_{B(Q_j^k)} |\psi(\eta)|d\sigma(\eta)\right)^2 \int_{Q_j^k} \varphi(\eta)\omega(\eta) d\sigma(\eta)\\&\lesssim  
[\omega]_{A_3} \int_{\bB} (M_{\omega^{-1/2}}(\psi\omega^{1/2})(\eta))^2 M_\omega (\varphi)(\eta)d\sigma(\eta)\\
&\lesssim [\omega]_{A_3}\|M_{\omega^{-1/2}}(\psi\omega^{1/2})^2\|_{L^{3/2}(\omega^{-1/2})}^2 \|M_\omega (\varphi)\|_{L^{3}(\omega)}\\
&\lesssim [\omega]_{A_3}\|\psi\|_{L^3(\omega)}^2.
\end{align*}
\end{proof}

\begin{lem}\label{lem:TlS}
For each $l\geq 1$, 
\begin{equation}\label{eqn:TlS}
\|{\mathcal T}_{l}^{\mathcal S}(\psi)\|_{L^3(\omega)}\lesssim l^{1/2}[\omega]_{A_3}^{1/2}\|\psi\|_{L^3(\omega)}.
\end{equation}

\end{lem}

\begin{proof}
If $l\geq 1$, we have that
$$\|{\mathcal T}_{l}^{\mathcal S}(\psi)\|_{L^3(\omega)}=\| ({\mathcal T}_{l}^{\mathcal S}(\psi))^2\|_{L^{3/2}(\omega)}^{1/2}.$$ 
Thus 
\begin{equation}\begin{split}\label{eqn:duality}
\| ({\mathcal T}_{l}^{\mathcal S}(\psi))^2\|_{L^{3/2}(\omega)}
&=\sup_{\|\varphi\|_{L^{3}(\omega^{-2})}\leq 1} \int_{\bB} {\mathcal T}_{l}^{\mathcal S}(\psi)^2(\eta)\varphi(\eta)d\sigma(\eta) \\
&=\sup_{\|\varphi\|_{L^{3}(\omega^{-2})}\leq 1}\int_{\bB}{\mathcal M}_l^{\mathcal S} (\psi,\varphi)(\eta)\psi(\eta)d\sigma(\eta),
\end{split}\end{equation}
where
$$
  {\mathcal M}_l^{\mathcal S} (\psi,\varphi)=\sum_{Q_j^k\in{\mathcal S}(Q)} \psi_{2^lB(Q_j^k)} \left( \frac1{|2^l B(Q_j^k)|} 
\int_{Q_j^k} \varphi\right)\Chi_{2^lB(Q_j^k)}(\eta).$$

Using the existence of adjacent families of cubes ${\mathcal D}^i$, $i=1,\cdots, M$ as in Proposition \ref{prop:dyadicdec}, 
the cubes $Q_j^k$ can be distributed in disjoint families ${\mathcal S}^i\in {\mathcal D}^i$ 
such that for any $Q_j^k\in {\mathcal S}^i$ there exists a dyadic cube $P_{j,k}^{l,i}\in{\mathcal D}^i$ 
with $2^l B(Q_j^k)\subset P_{j,k}^{l,i}$ and $l_{P_{j,k}^{l,i}}\lesssim 2^l l_{Q_j^k}$.
Thus
$$
{\mathcal M}_l^{\mathcal S} (\psi,\varphi)(\zeta)
\lesssim \sum_{i=1}^L {\mathcal M}_{i,l}^{\mathcal S} (\psi,\varphi)(\zeta),
$$
where
$$
{\mathcal M}_{l,i}^{\mathcal S} (\psi,\varphi)(\zeta)
=\sum_{Q_j^k\in {\mathcal S}_i}\psi_{P_{j,k}^{l,i}} \left( \frac1{|P_{j,k}^{l,i}|}\int_{Q_j^k} \varphi\right) \Chi_{P_{j,k}^{l,i}} (\zeta).
$$

The following lemma, for $\R^n$, can be found in \cite{Da-Le-Pe}.
\begin{lem}\label{lem:DamianLernerPerez}
If the sum ${\mathcal M}_{l,i}^{\mathcal S} (\psi,\varphi)$ is finite, there exists a finite number of cubes $Q_\nu\in{\mathcal D}^i$ 
covering its support and such that for any cube $Q_\nu$, there exist two families of sparse cubes 
${\mathcal S}^{i,1}$, ${\mathcal S}^{i,2}$ of ${\mathcal D}^i$, $i=1,\cdots, M$ satisfying that for $\zeta\in Q_\nu$,
$$
{\mathcal M}_{l,i}^{\mathcal S} (\psi,\varphi)(\zeta) 
\lesssim l\sum_{k=1}^2 \sum_{Q_j^k\in {\mathcal S}_{i,k}} \psi_{Q_j^k}\varphi_{Q_j^k} \Chi_{Q_j^k}(\zeta).
$$
\end{lem}
\begin{proof}
The proof of this lemma is, basically, an application of Lerner's decomposition  and the estimate
$$
\|{\mathcal F}_{l}(\Psi)\|_{L^{1,\infty}} \lesssim l\|\Psi\|_{L^1},
$$
where 
$$
{\mathcal F}_{l}(\Psi)= \sum_{j,k} \left( \frac1{|P_{j,k}^{l,i}|} \int_{Q_j^k}\Psi(\eta)d\sigma(\eta)\right)\Chi_{P_{j,k}^{l,i}},
$$
which can be found in Lemma 3,2 in \cite{Le2}. We remark that both constructions can be adapted to the framework 
of homogeneous spaces (see Remark 4.22 and Lemma 6.5 in \cite{An-Va}). 
In consequence, the nonisotropic version of Lemma \ref{lem:DamianLernerPerez} for the unit sphere holds.
\end{proof}

Now we can finish the proof of Lemma \ref{lem:TlS}, i.e. the estimate 
of $\|{\mathcal T}_l^{\mathcal S}(\psi)\|_{L^3(\omega)}$.

 Using the duality expression obtained in \eqref{eqn:duality},  Lemma \ref{lem:DamianLernerPerez} 
and H\"older's inequality, we have that
\begin{align*}
\int_{Q_\nu}{\mathcal M}_{l,i}^{\mathcal S} (\psi,\varphi)(\eta)\varphi(\eta)d\sigma(\eta)&\lesssim 
l \sum_{k=1,2} \sum_{Q_j^k\in{\mathcal S}_{i,k}} \left( \psi_{Q_j^k}\right)^2\int_{Q_j^k}\varphi(\eta)d\sigma(\eta)\\
&\lesssim
l \sum_{k=1,2}\int_{\bB}\left({\mathcal T}_{0}^{{\mathcal S}_{i,k}}(\psi)\right)^2(\eta)\varphi(\eta)d\sigma(\eta).
\end{align*}
Summing up over $Q_\nu$, and using \eqref{eqn:duality} and Lemma \ref{thm.lernerdec}, we deduce that
\begin{align*}
\|{\mathcal T}_{l}^{\mathcal S}(\psi)\|_{L^3(\omega)}=\|{\mathcal T}_l^{\mathcal S}(\psi)^2\|_{L^{3/2}(\omega)}^{1/2}
&\lesssim
\left(l \max_{1\leq i\leq M}\sup_{{\mathcal S}\in {\mathcal D}^i}\|{\mathcal T}_{0}^{\mathcal S}(\psi)^2\|_{L^3(\omega)}\right)^{1/2}\\
&\lesssim l^{1/2} [\omega]_{A_3}^{1/2} \|\psi\|_{L^3(\omega)}.
\end{align*}
\end{proof}

We can now finish the proof of Theorem \ref{thm:Bergman}. By the last lemmas we have that
\begin{align*}
\|{\mathcal T(\psi)}\|_{L^3(\omega)} 
&\leq \sum_{l^\geq 0} \frac1{2^{l/4}} \|{\mathcal T}_{l}^{\mathcal S}(\psi)\|_{L^3(\omega)} \\
&\lesssim \sum_{m^\geq 0} \frac{l^{1/2}}{2^{l/4}} [\omega]_{A_3}^{1/2}\|\psi\|_{L^3(\omega)}
\lesssim [\omega]^{1/2}_{A_3} \|\psi\|_{L^3(\omega)}.
\end{align*}

Thus, using the estimate \eqref{eqn:lernerthm2} and the continuity of $M(\psi)$, we  obtain
$$ 
\|\left(G(\psi)^2-m_Q(G(\psi)^2)\right)^{1/2}\|_{L^{3}(\omega)}
\lesssim [\omega]_{A_3}^\frac12\|\psi\|_{L^3(\omega)}.
$$

Hence
\begin{align*}
\|G(\psi)\|_{L^3(\omega)}&=\|G(\psi)^2\|_{L^{3/2}(\omega)}^{1/2} \\
&\lesssim \|G(\psi)^2-m_Q(G(\psi)^2)\|_{L^{3/2}(\omega)}^{1/2}+ \|m_Q(G(\psi)^2)\|_{L^{3/2}(\omega)}^{1/2}\\
&\lesssim  [\omega]^{1/2}_{A_3} \|\psi\|_{L^3(\omega)}+
\|m_Q(G(\psi)^2)\|_{L^{3/2}(\omega)}^{1/2}.
\end{align*}

Let us check that we also have that 
$$\|m_Q(G(\psi)^2)\|_{L^{3/2}(\omega)}^{1/2}\lesssim  [\omega]^{1/2}_{A_3} \|\psi\|_{L^3(\omega)}.$$
Indeed,
\begin{align*}
m_{Q}((G(\psi))^2)^{1/2} &\leq \left(\left( G(\psi)^2\Chi_Q\right)^*(|Q|/2)\right)^{1/2}
= \left( G(\psi)\Chi_Q\right)^*(|Q|/2) \\
&\lesssim \frac1{|Q|}\int_{Q}|\psi(\zeta)|d\sigma(\zeta).
\end{align*}

Consequently,
$$
\left( \int_{Q} m_{\bB}((G(\psi))^2)^{3/2}\omega(\zeta) d\sigma(\zeta)\right)^{1/3}
\lesssim \omega({Q})^{1/3} \frac1{|Q|}\int_{Q} |\psi(\zeta)|d\sigma(\zeta).
$$

But
\begin{align*}
\int_{Q} |\psi(\zeta)|d\sigma(\zeta)
&\leq \left(\int_{Q} |\psi(\zeta)|^3\omega(\zeta) d\sigma(\zeta) \right)^{1/3} 
\left(\int_{Q} \omega^{-1/2}(\zeta) d\sigma(\zeta) \right)^{2/3}\\
&=\|\psi\|_{L^3(\omega)} \left(\int_{Q} \omega^{-1/2} d\sigma \right)^{2/3}.
\end{align*}
Thus, 
\begin{equation}\label{eqn:lernerthm3}\begin{split}
&
\left( \int_{Q} m_{Q}((G(\psi))^2)^{3/2}(\zeta)\omega (\zeta)d\sigma(\zeta)\right)^{1/3}\\
&
\lesssim \|\psi\|_{L^3(\omega)} \frac{\omega({Q})^{1/3}}{|Q|} \left(\int_{Q} \omega^{-1/2} (\zeta)d\sigma (\zeta)\right)^{2/3}\\
&
\lesssim
\|\psi\|_{L^3(\omega)}  [\omega]_{A_3}^{1/3}\leq \|\psi\|_{L^3(\omega)} [\omega]_{A_3}^{1/2}.
\end{split}\end{equation}

Finally, applying Theorem \ref{thm:duoandikoetxea}, we obtain:

$$
\|{\mathcal C}(\psi)\|_{F^{p,2}_0(\omega)}=\|G(\psi)\|_{L^p(\omega)} 
\lesssim [\omega]^{\max\{1/2,1/(p-1)\}}_{A_p}\|\psi\|_{L^p(\omega)},
$$
which ends the proof of the theorem.\qed

\begin{rem}\label{rem:equivnormsC}
In Section \ref{sec:estimate} we have shown that
$$
\left\|(1-|z|^2)\left( I+ \frac{R}n\right)\Ca(\psi)\right\|_{L^{p,2}(\omega)}
\lesssim [\omega]^{\max\{1/2,1/(p-1)\}}_{A_p}\|\psi\|_{L^p(\omega)}.
$$
  However, from this estimate we can obtain the analogous estimate for the operator 
	$(1-|z|^2)\left( \alpha I+ \beta R\right)\Ca$ with $\alpha,\beta \in \R$.	Indeed,  assuming the above estimate, it is enough to show that
$$
\|(1-|z|^2)\Ca(\psi)\|_{L^{p,2}(\omega)}\lesssim [\omega]^{\max\{1/2,1/(p-1)\}}_{A_p}\|\psi\|_{L^p(\omega)}.
$$	
	And this is an immediate consequence of the relation
\begin{align*}
\Ca(\psi)(z)&=\int_\bB\frac{\psi(\z)}{(1-z\overline \z)^{n+1}}d\sigma(\z)
+\sum_{j=1}^n z_j \int_\bB\frac{\overline \z_j\psi(\z)}{(1-z\overline \z)^{n+1}}d\sigma(\z)\\
&=\left(I+\frac{R}{n}\right)\Ca(\psi)(z)+\sum_{j=1}^n z_j\left(I+\frac{R}{n}\right)\Ca(\overline\z_j\psi)(z),
\end{align*}
and the fact that $\|\overline{\zeta_j}\psi\|_{L^p(\omega)}\leq  \|\psi\|_{L^p(\omega)}$ .
\end{rem}

\section{Proof of Theorem \ref{cor:Bergman} and of the sharpness in Theorem \ref{thm:Bergman}}\label{sec:sharpness}

In order to prove the sharpness of the estimate  
$$
\|\Be:L^{p,2}(\omega)\to H^{p}(\omega)\|
\le C(p,n)[\omega]_p^{\max\{1/(2(p-1)),1\}}=C(p,n)\max\{[\omega]_{A_p},[\omega']_{A_{p'}}^{1/2}\},
$$
we  use the techniques in \cite{Fe-Pi} (see also  \cite{Lu-Pe-Re}). They are based in the following lemma,
whose proof follows from the Rubio de Francia algorithm.

\begin{lem}\label{lem:Du}
Let $1\le p_0<\infty$ and let $C,\beta>0$.  If  the pair $(\varphi,\psi)$  of nonnegative functions  satisfies
$$
\left(\int_{\bB} \psi^{p_0}\omega d\sigma\right)^{1/p_0}
\le C [\omega]_{A_1}^\beta \left(\int_{\bB} \varphi^{p_0}\omega d\sigma\right)^{1/p_0}\qquad\text{for all $\omega\in A_1$,}
$$
 then for any $p>p_0$ there exists a constant $C'=C'(n,\beta, p_0,C)$, such that 
$$
\|\psi\|_{L^{p}}\le C' p^{\beta} \|\varphi\|_{L^{p}}.
$$

Consequently, if the power $\beta$ of $p$ is sharp, the power $\beta$ of $[\omega]_{A_1}$ is also sharp.
\end{lem}

\begin{proof}
Let $q=p/p_0>1$. By duality, 
$$
\|\psi\|_{L^p}^{p_0}=\sup_{\|\phi\|_{L^{q'}}=1}\left|\int_\bB |\psi|^{p_0} \phi d\sigma\right|.
$$

Assume that $\phi\ge 0$. By the Rubio the Francia algorithm, the function 
$$
\omega(\z)=\sum_{k=0}^\infty \frac{M^k(\phi)(\z)}{(2\|M\|_{L^{q'}})^k}\qquad \text {($M^k$ denotes the $k$-th iterate of $M$)}
$$
satisfies $\phi (\z)\le \omega(\z)$ a.e., $\|\omega\|_{L^{q'}}\le 2\|\phi\|_{L^{q'}}=2$ 
and $[\omega]_{A_1}\le 2 \|M\|_{L^{{q'}}}\le c q\le  c p$ for some $c>0$.

Thus,
\begin{align*} 
\int_\bB |\psi|^{p_0} \phi d\sigma
&\le \int_S |\psi|^{p_0} \omega d\sigma 
\le C^{p_0} [\omega]_{A_1}^{\beta p_0}\int_\bB |\varphi|^{p_0} \omega d\sigma\\
&\le C^{p_0} [\omega]_{A_1}^{\beta p_0}\|\varphi\|_{L^{p}}^{p_0}\|\omega\|_{L^{q'}}
\le 2 C^{p_0} c^{\beta p_0} p^{\beta p_0} \|\varphi\|_{L^{p}}^{p_0}.
\end{align*}
This ends the proof.
\end{proof}

\begin{cor}\label{cor:Du}
Let $1<p_0<\infty$. If there exist positive constants $C$ and $\beta$ such that for any $\omega\in A_{p_0}$
$$
\|\Be:L^{p_0,2}(\omega)\to H^{p_0}(\omega)\|\le C [\omega]_{A_{p_0}}^\beta,
$$
then for any $1<p<\infty$, there exists $C'>0$ which does not depend on $p$, such that 
$$
\|\Be:L^{p,2}\to H^{p}\|\le C' \max\{p^\beta, (p')^{\beta(p_0-1)}\}.
$$
\end{cor}

\begin{proof}
First we prove the case $p>p_0$, that is  $\|\Be:L^{p,2}\to H^{p}\|\le C' p^\beta$.

Let $C_c(\B)$ be the space of continuous functions with compact support on $\B$. 
This space is in $L^{p_0,2}(\omega)$ for any $\omega\in A_{p_0}$ and it is dense in $L^{p,2}$ for every $1<p<\infty$.

For each $\omega\in A_1$, since  $[\omega]_{A_{p_0}}\le [\omega]_{A_{1}}$, we obtain 
$$
\|\Be:L^{p_0,2}(\omega)\to H^{p_0}(\omega)\|\le C[\omega]_{A_{p_0}}^\beta\le C[\omega]_{A_{1}}^\beta.
$$
Hence,  Lemma \ref{lem:Du} applied to the functions 
$$
\varphi(\z)=\left(\int_0^1 |\vartheta(r\z)|^2 \frac{2n r^{2n-1}}{1-r^2}dr\right)^{1/2}\quad\text{and}\quad \psi(\z)
=\Be(\vartheta)(\z),\qquad \vartheta\in C_c(\B),
$$
gives $\|\Be:L^{p,2}\to H^{p}\|\le C' p^\beta$ for any $p>p_0$.

Now we consider the case $1<p<p_0$. 
By Proposition \ref{lem:adjBergman}
$$
\|\Be:L^{p_0,2}(\omega)\to H^{p_0}(\omega)\|= \|\Ca:L^{p^{\prime}_0}(\omega')\to F^{p'_0,2}_0(\omega')\|
\le C [\omega']_{A_{p^\prime_0}}^\gamma
$$
with $\gamma=\beta(p^\prime_0-1)$.

Then,  for any $\omega'\in A_1$ we have  
$
\|\Ca:L^{p^{\prime}_0}(\omega')\to F^{p^\prime_0,2}_0(\omega')\|\le C[\omega']_{A_{1}}^\gamma.
$

Note that  $C(\bB)$, the space of continuous functions on $\bB$,  
is in $L^{p_0^\prime}(\omega')$ for any $\omega'\in A_1$ and it is dense in $L^{p'}$ for any $1<p'<\infty$.
Hence,  the above estimate and Lemma \ref{lem:Du} applied to the functions $\varphi\in C(\bB)$ and 
$$
\psi(\z)=
\left(\int_0^1 (1-r^2)\left|\left(I+\frac{R}{n}\right)\Ca(\varphi)(r\z)\right|^2 r^{2n-1}dr\right)^{1/2},
$$
gives $\|\Ca:L^{p'}\to F^{p',2}_0\| \le C' (p')^\gamma$.

The  relation $\|\Be:L^{p,2}\to H^{p}\|=\|\Ca:L^{p'}\to F^{p',2}_0\|$ finishes the proof.
\end{proof}

By Theorem \ref{thm:Bergman}, the hypotheses in the above corollary are true for $p_0=3/2$ and $\beta=1$. 
Thus, we have:

\begin{cor}\label{cor:estimthm12}
There exists $C>0$ such for any $1<p<\infty$ we have
$$
\|\Be:L^{p,2}\to H^{p}\|\approx \|\Ca:L^{p'}\to F^{p',2}_0\|\le C \max\{p, \sqrt{p'}\}.
$$
\end{cor}

In order to prove that the exponents $\beta=1$ and $\gamma=1/2$ cannot be replaced by any smaller one, 
we consider the function  $f(z)=\left(\frac{1+z}{1-z}\right)^{\alpha}$, with $0<\alpha<1$. 
This function   was used by several authors  to estimate the norms of some  classical operators.
For instance in \cite{Ho-Ver}, the authors use this function to prove that 
\begin{equation} \label{eqn:Cauchynorm}
\|\Ca:L^p\to H^p\|=\frac{1}{\sin(\pi/p)}\approx \max\{p,p'\}.
\end{equation}

The next lemma states the properties of these functions that we will need.

\begin{lem} \label{lem:f1}
Let $1<p<\infty$ and  $0<\delta<1$, let $f_{\delta,p}(z)=\left(\frac{1+z}{1-z}\right)^{\delta/p}$. 
Denote by $u_{\delta,p}$ and $v_{\delta,p}$ its  real and imaginary parts, respectively. 

Then, for each $p$ there exists $\delta_p>0$ such that for any $\delta_p<\delta <1$, we have:
\begin{enumerate}
\item \label{item:f11} For $0<\theta<2\pi$, $|v_{\delta,p}(e^{i\theta})|=\tan \frac{\delta \pi}{2p} \,\,u_{\delta,p}(e^{i\theta})$.
\item \label{item:f12} $\|f_{\delta,p}\|_{H^{p}}\approx \sqrt{p}\|f_{\delta,p}\|_{F^{p,2}_0}\approx \frac {1}{(1-\delta)^{1/p}}.$
\item \label{item:f13} If $1<p\le 2$, then $\|\Ca(u_{\delta,p})\|_{H^{p}}\approx p' \|u_{\delta,p}\|_{L^{p}}$.
\item \label{item:f14} If $p\ge 2$, then $ \|\Ca(v_{\delta,p})\|_{H^{p}}\approx p\|v_{\delta,p}\|_{L^{p}} $.
\end{enumerate}

Furthermore, the constants in the above equivalences do not depend on $p$ and $\delta$.
\end{lem}

\begin{proof}
In order to simplify the notations we will write $f$, $u $ and $v$ instead of $f_{\delta,p}$, $u_{\delta,p} $ and $v_{\delta,p}$, respectively.

 Assertion \eqref{item:f11} follows easily from the fact that  for $0<\theta<2\pi$,  $Re\,\frac{1+e^{i\theta}}{1-e^{i\theta}}=0$.

Let us prove \eqref{item:f12}. Since $|1-e^{i\theta}|\approx |\theta|$, we have 
$$
\|f\|_{H^p}\approx 1+\left(\int_0^{\pi/2} \theta^{-\delta}d\theta\right)^{1/p}\approx \frac{1}{(1-\delta)^{1/p}}.
$$

Now we estimate  the norm of $f$ in $F^{p,2}_0$, that is the norm of $(1-|z|^2)(I+R)f(z)$ on $L^{p,2}$.  
In order to obtain this estimate we prove that for $\delta$ near to 1, 
then the functions  $g(z)=(1-|z|^2)Rf(z)$ and $h(z)=(1-|z|^2)f(z)$  satisfy
$$
\|g\|_{L^{p,2}}\approx\frac{1}{\sqrt{p}} \frac{1}{(1-\delta)^{1/p}}\quad\text{and}\quad 
\|h\|_{L^{p,2}}\lesssim 1,
$$
with constants which do not depends on $p$ and $\delta$.
Combining these results with    
$$
\|g\|_{L^{p,2}}-\|h\|_{L^{p,2}}\le \|f\|_{F^{p,2}_0}\le \|g\|_{L^{p,2}}+\|h\|_{L^{p,2}}
$$
we obtain \eqref{item:f12}.

Let us prove these norm estimates of the functions $g$ and $h$. 
\begin{align*}
\|g\|_{L^{p,2}}
&\approx \frac{\delta}p\left(\int_{-\pi}^{\pi}
\left(\int_0^1(1-r^2)\frac{|1+r e^{i\theta}|^{2\delta/p-2}}{|1-re^{i\theta}|^{2\delta/p+2}}r^2dr\right)^{p/2}d\theta\right)^{1/p}.
\end{align*}

It is easy to check from the equivalences    $|1-re^{i\theta}|^2=(1-r)^2+2r^2(1-\cos\theta)\approx (1-r)^2+\theta^2$ 
for $0\le\theta\le\pi/2$ and analogously for $\pi/2\le\theta\le\pi$, $|1+re^{i\theta}|^2\approx (1-r)^2+(\pi-\theta)^2$, that  
\begin{equation}\label{eqn:f1}\begin{split}
\|g\|_{L^{p,2}}&\approx 1+ \frac{1}{p}\left(\int_{0}^{1}\left(\int_0^1\frac{2t}{(t^2+s^2)^{\delta/p+1}}dt\right)^{p/2}ds \right)^{1/p}\\
&\approx 1+ \frac{1}{\sqrt{p}}\left(\int_{0}^{1}\left(\frac {1}{s^{2\delta/p}}-\frac {1}{(1+s^2)^{\delta/p}}\right)^{p/2}ds\right)^{1/p}.
\end{split}\end{equation}
Thus, for $\delta>\delta_p=1-p^{-p/2}$, 
\begin{align*}
\|g\|_{L^{p,2}}
&\lesssim 1+ \frac{1}{\sqrt{p}}\left(\int_{0}^{1}\frac {d s}{s^{\delta}}\right)^{1/p}\approx \frac{1}{\sqrt{p}} \frac{1}{(1-\delta)^{1/p}}.
\end{align*}

Conversely, 
since  $\sqrt{a-b}\ge \sqrt{a}-\sqrt{b}$ for any $0<b<a$, \eqref{eqn:f1} and the triangular inequality gives 
$$
\|g\|_{L^{p,2}}\gtrsim 1+ 
\frac{1}{\sqrt{p}}\left(\int_{0}^{1}\frac {ds}{s^{\delta}}\right)^{1/p}
-\frac{1}{\sqrt{p}}\left(\int_{0}^{1}\frac {ds}{(1+s)^{\delta}}\right)^{1/p} \approx \frac{1}{\sqrt{p}} \frac{1}{(1-\delta)^{1/p}}.
$$

The proof of $\|h\|_{L^{p,2}}\lesssim 1$ is easier.
For  $p\ge 2$ follows from the fact that $|h(z)|^2\lesssim 1-|z|^2$ and for $1<p<2$ from 
$|h(z)|^2\lesssim (1-|z|^2)^2/|1-z|^2$.

In order to prove assertion   \eqref{item:f13}, note that  if $1-\delta<p-1\le 1$ then 
$\tan\frac{\delta\pi}{2p}\approx \left(1-\frac{\delta}{p}\right)^{-1}\approx p'$. 
Hence, assertions \eqref{item:f12} and \eqref{item:f11} give
 $$
\|2\Ca(u)(z)\|_{H^p}=\|f(z)+\overline {f(0)}\|_{H^p}\approx \|u\|_{L^p}+\|v\|_{L^p}\approx p'\|u\|_{L^p}.
$$

Analogously, \eqref{item:f14} follows from the fact that for $p\ge 2$, $\tan\frac{\delta\pi}{2p}\approx \frac{1}{p}$, and 
$$
\|2\Ca(v)(z)\|_{H^p}=\|f(z)-\overline {f(0)}\|_{H^p}\approx \|u\|_{L^p}+\|v\|_{L^p}\approx p\|v\|_{L^p}.
$$
This concludes the proof.
\end{proof}

\subsection{Proof of Theorem \ref{cor:Bergman}}

\begin{proof}
By  Corollary \ref{cor:estimthm12} we have that 
$$
\|\Be:L^{p,2}\to H^{p}\|=\|\Ca:L^{p'}\to F^{p',2}_0\|\le c(n)  \max\{p, \sqrt{p'}\}.
$$

In order to prove that this estimate is sharp, we consider the case $n=1$. Assume  $1<p\le 3/2$.
Let $f=f_{\delta,p'}$ as in Lemma \ref{lem:f1} and $v=v_{\delta,p'}$ its imaginary part. 
Then,  we have
$$
\| \mathcal{C}(v)\|_{F^{p',2}_0}\approx  \frac{1}{\sqrt{p'}}\|f\|_{F^{p,2}_0}\approx \sqrt{p'} \|v\|_{L^{p'}}.
$$
Thus,  $\|\Ca:L^{p'}\to F^{p',2}_0\|\gtrsim C \sqrt{p'}$.

Now we consider the case $p>3/2$. Since, $1<p'<3$,  for $\varphi\in L^{p'}$ the norms of $\mathcal{C}(\varphi)$ on $H^{p'}$ 
and on $F^{p',2}_0$ are equivalent with constants that do not depend on $p$. 
Since, by \eqref{eqn:Cauchynorm}, the norm of  $\mathcal{C}:L^{p'}\to H^{p'}$ is equivalent to $p$, we conclude the proof.
\end{proof}

\subsection{Proof of the sharpness in Theorem \ref{thm:Bergman}}


\begin{proof}
We prove that there is not any $\lambda<1$ such that 
\begin{equation}\label{cor:sharp}
\|\Be :L^{p,2}(\omega)\to H^p(\omega)\|\leq C(p,n) [\omega]_{A_p}^{\lambda\max \{1,1/(2(p-1))\}}.
\end{equation}

Assume that \eqref{cor:sharp} is satisfied for some $p_0$ and some $\lambda<1$. 
Then, by Corollary \ref{cor:Du}, for $1<p<\infty$ we have 
$$
\|\Be :L^{p,2}\to H^p\|\leq C' \max\{p^\beta,(p')^{\beta(p_0-1)}\},\quad \beta=\lambda\max \{1,1/(2(p_0-1))\}.
$$
 
If $p_0> 3/2$, then $\beta=\lambda$ and thus $\|\Be :L^{p,2}\to H^p\|\leq C' p^\lambda$ for any $p>3/2$. 
This is not possible by Theorem \ref{cor:Bergman}.

If $p_0\le 3/2$, then $\beta(p_0-1)=\lambda/2$ and thus $\|\Be :L^{p,2}\to H^p\|\leq C' (p')^{\lambda/2}$ for any $1<p<3/2$. 
As above,  Theorem \ref{cor:Bergman} gives that this is not possible.
\end{proof}

\section{Proof of Theorems \ref{thm:Bergman2} and \ref{cor:Bergman3}} \label{sec:thmbergman2}

In Section \ref{sec:necessary} it is proved that if $\Be$ is  bounded from 
$L^{p,2}(\omega)$ to $F^{p,2}_0(\omega)$, then $\omega\in A_p$.

Conversely, if $\omega\in A_p$, $p>1$, then $H^p(\omega)$ and $F^{p,2}_0(\omega)$ are isomorphic. 
Hence,  Theorem \ref{thm:Bergman} 
ensures then that $\|\Be:L^{p,2}(\omega)\to F^{p,2}_0(\omega)\|$ is finite.

In order to obtain norm-estimates for this  operator, let
$$\Qo(\varphi)(z)=(1-|z|^2)\left(I+\frac{R}{n}\right)\Be(\varphi)(z).$$

If $\varphi$ and $\psi$ are smooth functions on $\B$, from 
$$\left(I+\frac{R}{n}\right)\Be(z,w)=\overline{\left(I+\frac{R}{n}\right)\Be(w,z)}$$
and Fubini's theorem, we have 
$$
\langle \Qo(\varphi),\psi\rangle_{\B}=\langle \varphi,\Qo\psi\rangle_{\B}, 
$$
where $\langle \cdot,\cdot \rangle_{\B}$ denotes the pairing given in Proposition \ref{prop:dualLpq}.
Thus, 
\begin{equation}\label{eqn:dualQ}\begin{split}
\|\Be:L^{p,2}(\omega)\to F^{p,2}_0(\omega)\|&=\|\Qo:L^{p,2}(\omega)\to L^{p,2}(\omega)\|\\
&=\|\Qo:L^{p',2}(\omega')\to L^{p',2}(\omega')\|\\
&=\|\Be:L^{p',2}(\omega')\to F^{p',2}_0(\omega')\|.
\end{split}\end{equation}

Consider the homogeneous space $(\B',d,\nu)$ where $\B'=\overline{\B}\setminus \{0\}$,
$d$ denotes the quasimetric
$$
d(z,w)=\max\{|\,|z|-|w|\,|,|1-z^*{\overline {w}}^*|\}, \qquad z^*=z/|z|, \quad w^*=w/|w|,
$$
and $\nu$ is the volume measure on $\B'$.
Denote by $\Delta(z,r)$ the balls with respect to the metric $d$. 
Observe that  if $\z\in \bB$, then  the ball $\Delta(\z,r)$ 
coincides with the square $S_{r\z}$ introduced in Proposition \ref{prop:neccond}.

A weight $\Omega\in L^1(\B')$ is in the Muckenhoupt class $A_2(\B')$ with respect to the homogeneous space $\B'$ if  
$$
[\Omega]_{A^2(\B')}=\sup_{\Delta}\frac{1}{\nu(\Delta)^2}\int_\Delta \Omega d\nu \,\int_\Delta \Omega^{-1}d\nu<\infty.
$$

\begin{lem}\label{lem:OmA2}
If $\omega\in A_2(\bB)$, then the weight $\Omega(z)=\omega(z/|z|)$, $z\ne 0$, is in $A_2(\B')$ and 
$[\Omega]_{A^2(\B')}\lesssim [\omega]_{A_2}$.
\end{lem}

\begin{proof}
By integration in polar coordinates 
\begin{align*}
\int_{\Delta(a,r)} \Omega(z) d\nu(z)
&\lesssim r \int_{\{\z\in\bB:\,d(a^*,\z)<r\}}\omega(\z)d\sigma(\z).
\end{align*}
Analogously
\begin{align*}
\int_{\Delta(a,r)} \Omega^{-1}(z) d\nu(z)
&\lesssim r \int_{\{\z\in\bB:\,d(a^*,\z)<r\}}\omega^{-1}(\z)d\sigma(\z).
\end{align*}
Since, $\sigma\{\z\in\bB:\,d(a^*,\z)<r\}\approx r^n$ and $\nu(\Delta(a,r))\approx r^{n+1}$, 
we obtain $[\Omega]_{A^2(\B)}\lesssim [\omega]_{A_2}$,
which concludes the proof.
\end{proof}

\begin{cor}\label{cor:CZ}
If $\omega\in A_2$, then 
$$
\|\Be:L^{2,2}(\omega)\to F^{2,2}_0(\omega)\|\le C [\omega]_{A_2}.
$$ 
\end{cor}

\begin{proof}
In \cite{An-Va}, it was proved that if $T$ is Calderon-Zygmund operator on a homogeneous space $X$, 
then for any $\Omega\in A_2(X)$ we have $\|T\|_{L^2(X,\Omega)}\le C(T,X)[\Omega]_{A_2(X)}$.

Observe that $ L^{2,2}(\omega)= L^2\left(\B',\frac{\Omega(z)}{1-|z|^2}d\nu(z)\right)$. 
Thus, the boundednes of  the operator $\Qo$ on $L^{2,2}(\omega)$ is equivalent to the boundedness 
of the Calderon-Zygmund operator 
$T:L^2(\B,\Omega)\to L^2(\B,\Omega)$ defined by 
$$
T(\varphi)(z)=\int_\B\varphi(w)(1-|w|^2)^{1/2}(1-|z|^2)^{1/2}\left(I+\frac{R}{n}\right)\frac{1}{(1-z\overline w)^{n+1}}d\nu(w).
$$
Applying the above mentioned result to $T$ and $X=\B'$, 
and using Lemma \ref{lem:OmA2} we obtain  the estimate.
\end{proof}

Using this estimate and  the extrapolation  Theorem \ref{thm:duoandikoetxea} we obtain:

\begin{thm}\label{thm:normBLtoF}
Let $1<p<\infty$. There exists a positive constant $C(p,n)$ such that  
$$
\|\Be:L^{p,2}(\omega)\to F^{p,2}_0(\omega)\|
\le C(p,n) [\omega]_{A_p}^{\max\{1,1/(p-1)\}}
=C(p,n)\max\{[\omega]_{A_p},[\omega']_{A_{p'}}\}
$$
for any $\omega\in A_p$.
\end{thm}

\begin{rem}\label{rem:equivnormsB}
Note that the same arguments used to prove 
$$
\|\Qo:L^{p,2}(\omega)\to L^{p,2}(\omega)\|\le C(p,n) [\omega]_{A_p}^{\max\{1,1/(p-1)\}}
$$
show that for any real numbers $\alpha$ and $\beta$,
$$
\|(1-|z|^2)\left(\alpha I+\beta R\right)\Be:L^{p,2}(\omega)\to L^{p,2}(\omega)\|\le C(p,n,\alpha,\beta) [\omega]_{A_p}^{\max\{1,1/(p-1)\}}.
$$
That is, if in the space $F^{p,2}_0(\omega)$ we consider the norm 
$$
\|f\|_{F^{p,2}_0(\omega)}=\|(1-|z|^2)(\alpha I+\beta R) f\|_{L^{p,2}(\omega)},
$$
 with $\alpha>0$ and $\beta>0$, we also obtain the same estimate  
$$
\|\Be:L^{p,2}(\omega)\to F^{p,2}_0(\omega)\|\lesssim [\omega]_{A_p}^{\max\{1,1/(p-1)\}}.
$$
\end{rem}

{\bf Proof of Theorem \ref{thm:Bergman2}}

\begin{proof}
In Section \ref{sec:necessary} it is proved that if $\Be$ is  bounded from 
$L^{p,2}(\omega)$ to $F^{p,2}_0(\omega)$, then $\omega\in A_p$.

The estimate 
$$
\|\Be:L^{p,2}(\omega)\to F^{p,2}_0(\omega) \|\lesssim [\omega]_{A_p}^{\max\{1,1/(p-1)\}}
$$
follows from  Theorem \ref{thm:normBLtoF}.  
\end{proof}

The following couple of lemmas show that it can not be obtained as an upper bound of the norm of $\Be$ in terms of $[\omega]_{A_p}^\lambda$ with $\lambda<1/p$.

\begin{lem}\label{lem:omf}
For  $1<p<\infty$, $0<\rho<1/2$  and $0<\delta<1$, let 
$\omega_\delta(e^{i\theta})=|1-e^{i\theta}|^{(p-1)(1-\delta)}$ and
$\varphi_{\delta}(re^{i\theta})=|1-e^{i\theta}|^{\delta-1}(1-r)\Chi_{\rho,1}$, 
where $\Chi_{\rho,e^{i\eta}}$ denotes the characteristic function of the square
$S_{\rho,e^{i\eta}}=\{z=r e^{i\theta}\in\D: 1-r<\rho, |1-e^{i(\theta-\eta)}|<\rho\}$.

We then have:
\begin{enumerate}
	\item \label{item:omf1} $[\omega_\delta]_{A_p}\approx \delta^{1-p}$.
	\item \label{item:omf2} $\|\varphi_\delta\|_{L^{p,2}(\omega_\delta)}\approx \delta^{-1/p}$.
	\item \label{item:omf3}$\|\varphi_\delta\|_{L^{1}}\approx \delta^{-1}$.
	\item \label{item:omf4}$\|(1-|z|^2)\Chi_{\rho,-1}(z)\|_{L^{p,2}(\omega_\delta)}\approx C$.
\end{enumerate}
\end{lem}

\begin{proof}
From  $\omega_\delta(e^{i\theta})\approx |\theta|^{(p-1)(1-\delta) }$ 
it is easy to check that 
$$
[\omega_\delta]_{A_p}\approx  \delta^{1-p}
$$

The remainder estimates follows from
\begin{align*}
&\|\varphi_\delta\|_{L^{p,2}(\omega_\delta)}\approx \left(\int_{-\rho}^{\rho}|\theta|^{\delta-1}d\theta\right)^{1/p}
\left(\int_{1-\rho}^1(1-r)dr\right)^{1/2}\approx \delta^{-1/p}.\\
&\|\varphi_\delta\|_{L^{1}} \approx \int_{-\rho}^{\rho}|\theta|^{\delta-1}d\theta \int_{1-\rho}^1(1-r)dr\approx \delta^{-1}.\\
&\|(1-|z|^2)\Chi_{\rho,-1}(z)\|_{L^{p,2}(\omega_\delta)}
\approx C.
\end{align*}
The constants in the last equivalences depends of $\rho$.
\end{proof}

\begin{lem} \label{lem:sharpBergman2}
Let  $\varphi_\delta$ and $\omega_\delta$ be as in Lemma \ref{lem:omf}.
Then 
$$
\|\Be:L^{p,2}(\omega_\delta)\to F^{p,2}_0(\omega_\delta) \|\gtrsim [\omega_\delta]_{A_p}^{1/p}.
$$
\end{lem}

\begin{proof}
For any $z\in S_{\rho,-1}$ and any $w\in S_{\rho}$, $|1-z\overline w|>1/2$ and consequently 
\begin{align*}
\|\Be(\varphi_\delta)\|_{L^{p,2}(\omega_\delta)}
\gtrsim  \|\Be(\varphi_\delta)\Chi_{S_{\rho,-1}}\|_{L^{p,2}(\omega_\delta)}
\gtrsim \|(1-|z|^2)\Chi_{-\rho}(z)\|_{L^{p,2}(\omega_\delta)}\|\varphi_\delta\|_{L^1(d\nu)}.
\end{align*}
By Lemma \ref{lem:omf}, the last term is equivalent to 
$$
\delta^{-1}\approx \|\varphi_\delta\|_{L^{p,2}(\omega_\delta)}[\omega_\delta]_{A_p}^{1/p}.
$$
Hence, $\|\Be:L^{p,2}(\omega_\delta)\to L^{p,2}(\omega_\delta) \|\gtrsim [\omega_\delta]_{A_p}^{1/p}$
 which concludes the proof.
\end{proof}

\subsection{Proof of Theorem \ref{cor:Bergman3}}

\begin{proof}
As we have already said in the introduction, for  $p>0$  the norms on the spaces $H^p$ and  
$F^{p,2}_0$ are equivalent. From this fact it is easy to check that for $1\le p\le 2$ this equivalence 
can be done by constants which do not depend   on $p$.  
Thus, by Theorem \ref{cor:Bergman} we have 
$$
\|\Be:L^{p,2}\to F^{p,2}_0\|\approx \|\Be:L^{p,2}\to H^p\|
\lesssim \sqrt{p'},
$$
and this estimate is sharp.

The case $p>2$ follows from   \eqref{eqn:dualQ} and the above result.
\end{proof}

\section{Proof of Theorem \ref{thm:Bergman3}}\label{sec:thmbergman3}

In order to prove the estimate in Theorem \ref{thm:Bergman3}, we need the following:

\begin{prop}\label{prop:BtoFA1}
For $2< p<\infty$ and $\omega\in A_{p/2}$,  
$$
\|\Qo:L^{p,2}(\omega)\to L^{p,2}(\omega)\|\le C(p,n) [\omega]_{A_{p/2}}^{1/2},
$$
where,  as in the above section, $\Qo=(1-|z|^2)\left(I+\frac{R}{n}\right)\Be$.
\end{prop}

\begin{proof}
Assume $0\le \varphi\in L^{p,2}(\omega)$ and denote by $|\Qo|$  the integral operator with kernel $|\Qo(z,w)|$.

Since $|\Qo|(1)\approx 1$, by H\"older's inequality,  we have
$$
|\Qo(\varphi)(z)|^2\lesssim |\Qo|(\varphi^2)(z) ,
$$
Using this fact, Fubini's theorem  and 
$|1-r\z \overline w|\approx 1-r+ |1-\z\overline w|$, we obtain
$$
\int_0^1 |\Qo|(\varphi^2)(r\z)\frac{dr}{1-r^2}
\lesssim \int_\B\frac{\varphi^2(w)}{|1-\z\overline w|^{n+1}}d\nu(w).
$$
Hence, duality $\left(L^{p/2}(\omega)\right)'=L^{(p/2)'}(\omega)$  and Fubini's theorem give
$$
\|\Qo(\varphi)\|_{L^{p,2}(\omega)}
\lesssim \sup_{\|\psi\|_{L^{(p/2)'}(\omega)}=1} 
\left(\int_\B\varphi^2(w)
\int_\bB \frac{|\psi(\z)|\omega(\z)}{|1-\z\overline w|^{n+1}}d\sigma(\z)d\nu(w)\right)^{1/2}.
$$
By \cite[Lemma 3]{Dr-Gra-Pe-Pe}, 
there exists a function $v\in L^{(p/2)'}(\omega)$ such that 
$$
v\ge |\psi|, \quad \|v\|_{L^{(p/2)'}(\omega)}\le 2 \|\psi\|_{L^{(p/2)'}(\omega)}\quad
\text{and}\quad [v\omega]_{A_1}\le 2 C(p/2)[\omega]_{A_{p/2}}
.$$
Thus, if $w=t\eta$, then, a.e $\eta\in\bB$, we have that
\begin{align*}
\int_\bB \frac{|\psi(\z)|\omega(\z)}{|1-\z\overline w|^{n+1}}d\sigma(\z)&
\le \frac{1}{1-t^2} M(v\omega)(\eta)
\le \frac{2}{1-t^2}[v\omega]_{A_1} \,v(\eta)\omega(\eta)\\
&\le C((p/2)') [\omega]_{A_{p/2}} \frac{v(\eta)\omega(\eta)}{1-t^2}.
\end{align*}
By H\"older's inequality and 
$\|v\|_{L^{(p/2)'}(\omega)}\le 2 \|\psi\|_{L^{(p/2)'}(\omega)}$, we obtain
$$
\int_\B\varphi^2(w)
\int_S \frac{|\psi(\z)|\omega(\z)}{|1-\z\overline w|^{n+1}}d\sigma(\z)d\nu(w)
\lesssim C((p/2)') [\omega]_{A_{p/2}} \|\varphi\|_{L^{p,2}(\omega)}^2,
$$
which concludes the proof.
\end{proof}

\subsection{Proof of Theorem \ref{thm:Bergman3}}
\begin{proof} We want to prove that:
\begin{enumerate}
	\item \label{item:duo1} If $2<p<\infty$ and  $\omega\in A_{1}$, then 
$\|\Be:L^{p,2}(\omega)\to F^{p,2}_0(\omega)\|\le C(p,n) [\omega]_{A_1}^{1/2}$.
\item \label{item:duo2} If $1<p< 2$ and $\omega' \in A_{1}$, then 
$\|\Be:L^{p,2}(\omega)\to F^{p,2}_0(\omega)\|\le C(p,n) [\omega']_{A_1}^{1/2}$.
\end{enumerate}


By \eqref{eqn:dualQ}, if $\omega\in A_p$, then 
$$
\|\Be:L^{p,2}(\omega)\to F^{p,2}_0(\omega)\|=\|\Qo:L^{p,2}(\omega)\to L^{p,2}(\omega)\|
=\|\Be:L^{p',2}(\omega')\to F^{p',2}_0(\omega')\|.
$$
Therefore, assertion \eqref{item:duo1} follows from Proposition \ref{prop:BtoFA1} 
and the fact that $[\omega]_{A_{p/2}}\le [\omega]_{A_1}$, $p>2$. 
Part \eqref{item:duo2} follows from  identity \eqref{eqn:dualQ} and part \eqref{item:duo1}. Indeed, 
if $\omega'\in A_1$, then $\omega\in A_p$ and 
$$
\|\Be:L^{p,2}(\omega)\to F^{p,2}_0(\omega)\|
=\|\Be:L^{p',2}(\omega')\to F^{p',2}_0(\omega')\|\le C(p,n) [\omega']_{A_1}^{1/2}.
$$

The sharpness of the above estimates follows from Lemma \ref{lem:Du}  and Theorem \ref{cor:Bergman3}. 
Indeed, for $p_0>2$, following the same arguments used to prove Corollary \ref{cor:Du}, we obtain that if 
$\|\Be:L^{p_0,2}(\omega)\to F^{p_0,2}_0(\omega)\|\le C(p,n) [\omega]_{A_1}^{\beta}$,  
then Lemma \ref{lem:Du} applied to the functions 
\begin{align*}
\varphi(\z)&=\left(\int_0^1 |\vartheta(r\z)|^2 \frac{2n r^{2n-1}}{1-r^2}dr\right)^{1/2}\quad\text{and},\\ 
\psi(\z)&=\left(\int_0^1 (1-r^2)\left(I+\frac{R}{n}\right)\Be(\vartheta)(r\z)\frac{2n r^{2n-1}}{1-r^2}dr\right)^{1/2},
\end{align*}
for $\vartheta\in C_c(\B)$, gives
$\|\Be:L^{p,2}\to F^{p,2}_0\|\le C(n) p^{\beta}$ for $p>p_0$. 
By Theorem \ref{cor:Bergman3} we have $\beta\ge 1/2$, which proves that the estimate is sharp.
\end{proof}


\begin{thebibliography}{12}
 
\bibitem{Ah-Br} Ahern, P.; Bruna, J.: 
Maximal and area integral characterizations of Hardy-Sobolev spaces in the unit ball of $C^n$. 
Rev. Mat. Iberoamericana  \textbf{4}  (1988),  no. 1, 123–153.

\bibitem{An-Va} Anderson, T.C. and  Vagharshakyan, A.: 
A simple proof of the sharp weighted estimate for Calderon-Zygmund operators on homogeneous spaces. 
J. Geom.Anal. \textbf{24} (2014), 1276-1297.

\bibitem{Bekolle} B\'ekoll\'e, D.: 
In\'egalit\'es \`a poids pour le projecteur de Bergman dans la boule unit\'e de $\C^n$. 
Stud. Math. \textbf{71} (1981), 305-323.

\bibitem{BenedekPanzone} Benedek, A.; Panzone, R.:
 The space $L^p$, with mixed norm. 
Duke Math. J. \textbf{ 28 } (1961), 301-324.

\bibitem{Bu} Buckley, S.M.: 
Estimates for operators on weighted spaces and reverse Jensen inequalities. 
Trans. Amer. Math. Soc. \textbf{340} (1993), no. 1. 253-272.

\bibitem{Ca-Or} Cascante, C.; Ortega, J.M.:  
Weak trace measures on Hardy-Sobolev spaces. 
Math. Res. Lett. \textbf{20} (2013), no. 2, 235-254.

\bibitem{Ca-Or4} Cascante, C.; Ortega, J.M.:  
Carleson measures for weighted Hardy-Sobolev spaces. 
Nagoya Math. J.  \textbf{186}  (2007), 29–68.

\bibitem{Co} Cohn, W.S.: 
The Bergman projection and vector-valued Hardy spaces.
Michigan Math. J.\textbf{44} (1997), no. 3, 509-528.

\bibitem{Co-Fe} Coifman, R.R.; Fefferman, C.:
Weighted norm inequalities for maximal functions and singular integrals.
Studia Math. \textbf{51} (1974), 241-250.

\bibitem{Chr} Christ, M.:   
A $T(b)$ theorem with remarks on analytic capacity and the Cauchy integral.
Colloq. Math., \textbf{60/61}, (1990), 601--628.

 \bibitem{CrU-Ma-Pe} Cruz-Uribe, D.; Martell, J. M.; Perez, C.: 
Sharp weighted estimates for classical operators. 
Adv. Math. \textbf{229} (2012), no. 1, 408-441.

\bibitem{Da-Le-Pe}   Damian, W.;  Lerner A.K.; Perez, C.: 
Sharp weighted bounds for multilinear maximal functions and Calderón-Zygmund operators. 
Preprint available at:  arXiv:1211.5115.

\bibitem{Dr-Gra-Pe-Pe} Dragi\v{c}evi\'{c}, O.; Grafakos, L.; Pereyra, M.C.; Petermichl, S.:
Extrapolation and sharp norm estimates for classical operators on weighted Lebesgue spaces.
Publ. Mat.  \textbf{49} (2005),  no. 1, 73–91. 

\bibitem{Du} Duoandikoetxea, J.: 
Extrapolation of weights revisited: new proofs and sharp bounds. 
J. Funct. Anal. \textbf{260} (2011), no. 6, 1886-1901.

\bibitem{Du2} Duoandikoetxea, J.:
Fourier analysis. 
Graduate Studies in Mathematics, \textbf{29}. A.M.S., Providence, RI, 2001. 
 
\bibitem{Fe-Pi}  Fefferman, R.; Pipher, J.: 
Multiparameter operators and sharp weighted inequalities. 
Amer. J. Math. \textbf{119} (1997), no. 2, 337–369.

\bibitem{Ho-Ver} Hollenbeck, B.; Verbitsky, I.E.: 
Best constants for the Riesz projection. 
J. Funct. Anal. \textbf{175} (2000), 370-392.

\bibitem{Hu-Mu-Whe} Hunt.; Muckenhoupt, B.; Wheeden,R.:
Weighted norm inequalities for the conjugate function and Hilbert transform. 
Trans. Amer. Math. Soc.  \textbf{176}  (1973), 227-251.

\bibitem{Hy} Hyt\"{o}nen, T.P.: 
The sharp weighted bound for general Calderón-Zygmund operators. 
Ann. of Math. (2) \textbf{175} (2012), no. 3, 1473-1506.

\bibitem{HyKa} Hyt\"{o}nen, T.P.; Kairema, A.: 
Systems of dyadic cubes in a doubling metric space. 
Colloq. Math. \textbf{126} (2012), no. 1, 1-33.

 \bibitem{JaTo} Jawerth, B.; Torchinsky, A.: 
The strong maximal function with respect to measures. 
Studia Math. \textbf{80} (1984), no. 3, 261-285. 

\bibitem{Lee-Rim}Lee, J.; Rim, K.S.:
Weighted norm inequalities for pluriharmonic conjugate functions. 
J. Math. Anal. Appl.  \textbf{268}  (2002),  no. 2, 707-717. 

\bibitem{Le0} Lerner, A.K.:  
 Sharp weighted norm inequalities for Littlewood-Paley operators and singular integrals. 
Adv. Math. \textbf{226} (2011), no. 5, 3912-3926

\bibitem{Le} Lerner, A.K.: 
On sharp aperture-weighted estimates for square functions. 
J. Fourier Anal. Appl. \textbf{20} (2014), no 4, 784-800.

\bibitem{Le2} Lerner, A.K.: 
A simple proof of the $A_2$ conjecture.   Int. Math. Res. Not. IMRN 2013, no 14, 3159-3170
Preprint available at: arXiv:1202.2824.  

\bibitem{Le3} Lerner, A.K.: 
A local mean oscillation decomposition and some of its applications.
Function spaces, approximation, inequalities and linearity, Lect. of the Spring School in Anal., Matfyzpress, Prague (2011), 71-106.

 \bibitem{Le4} Lerner, A.K.:  
A pointwise estimate for the local sharp maximal function with applications to singular integrals. 
Bull. London Math. Soc. \textbf{42} (2010) no. 5 843-856.

\bibitem{Lu-Pe-Re} Luque, T., Perez, C., Rela, E.: 
Optimal exponents in weighted estimates without examples. 
Preprint available at arXiv:1307-5642. 

\bibitem{Ka} Kairema, A.: 
Two-weight norm inequalities for potential type and maximal operators in a metric space. 
Publ. Mat. \textbf{57} (2013), no. 1, 3-56.

\bibitem{Lu} Luecking, D.H.: 
Representation and duality in weighted spaces of analytic functions. 
Indiana Univ. Math. J. \textbf{34} (1985), no. 2, 319-336.

\bibitem{Pav} Pavlovi\'{c}, M.:
On the Littlewood-Paley $g$-function and Calder\'on's area theorem.
Expo. Math. \textbf{ 31}  (2013),  no. 2, 169-195.

\bibitem{Po-Re} Pott, S.; Reguera, M. C.: 
Sharp B\'ekoll\'e estimates for the Bergman projection. 
J. Funct. Anal. \textbf{265} (2013), no. 12, 3233-3244.

\bibitem{Ru} Rudin, W.: 
Function theory in the unit ball of $\C^n$. 
Springer Verlag, New York, 1980.


\bibitem{Stein} Stein, E.M.: 
Harmonic analysis: real-variable methods, orthogonality, and oscillatory integrals. 
Princeton Mathematical Series, 43. Monographs in Harmonic Analysis III, 1993. 

\bibitem{tch} Tchoundja, E.: 
Carleson measures for the generalized Bergman spaces via a $T(1)$-type theorem.  
Ark. Mat. {\bf 46}, (2008), 377-406.


\bibitem{Wi} Wilson, M.: 
The intrinsic square function. 
Rev. Mat. Iberoam. \textbf{23} (2007), no. 3, 771-791.

 \end{thebibliography}
\end{document}